\documentclass[reqno,12pt]{amsart}
\usepackage[margin=1.0in]{geometry}
\usepackage[latin1]{inputenc}
\usepackage{subfigure}
\usepackage{rotating}
\usepackage[svgnames]{xcolor}
\usepackage{tikz}
\usepackage{url}
\usetikzlibrary{arrows}
\usepackage{subfigure}
\usepackage{amsmath,amsthm,enumerate,amssymb,amsbsy}
\usepackage{algorithm}
\usepackage{algorithmic}

\usepackage{comment}
\usepackage{xspace}
\usepackage[version-1-compatibility]{siunitx}

\usepackage{subfigure}

\usepackage{booktabs}
\usepackage{hyperref}

\author{V. Baldoni}
\address{Velleda Baldoni: Dipartimento di Matematica, Universit\`a degli studi di  Roma ``Tor Vergata'',
Via della ricerca scientifica 1, I-00133 Roma, Italy}
\email{baldoni@mat.uniroma2.it}
\author{N. Berline}
\address{Nicole Berline: Centre de Math\'ematiques Laurent Schwartz, \'Ecole Polytechnique, 91128 Palaiseau Cedex, France}
\email{nicole.berline@math.polytechnique.fr}
\author{J. A. De Loera}
\address{Jes\'us A. De Loera:  Department of
  Mathematics, University of California,
  Davis, One Shields Avenue, Davis, CA, 95616, USA}
\email{deloera@math.ucdavis.edu}
\author{B. E. Dutra}
\address{Brandon E. Dutra:  Department of
  Mathematics, University of California,
  Davis, One Shields Avenue, Davis, CA, 95616, USA}
\email{bedutra@ucdavis.edu}
\author{M. K\"oppe}
\address{Matthias~K\"oppe:  Department of
  Mathematics, University of California,
  Davis, One Shields Avenue, Davis, CA, 95616, USA}
\email{mkoeppe@math.ucdavis.edu}
\author{M. Vergne}
\address{Mich\`ele Vergne: Institut de Math\'ematiques de Jussieu, Th{\'e}orie des
  Groupes, Case 7012, 2 Place Jussieu, 75251 Paris Cedex 05, France}
 \email{vergne@math.jussieu.fr}

\title[Coefficients of the Denumerant]{Coefficients of Sylvester's Denumerant}


\theoremstyle{plain}
\newtheorem{theorem}{Theorem}[section]
\newtheorem{proposition}[theorem]{Proposition}
\newtheorem{lemma}[theorem]{Lemma}

\theoremstyle{definition}
\newtheorem{definition}[theorem]{Definition}
\newtheorem{corollary}[theorem]{Corollary}
\newtheorem{example}[theorem]{Example}
\newtheorem{remark}[theorem]{Remark}

\numberwithin{equation}{section}

\newcommand{\C}{{\mathbb C}}
\newcommand{\R}{{\mathbb R}}
\newcommand{\Z}{{\mathbb Z}}

\let\ve=\mathbf
\newcommand\smallstep[1]{\fractional{-#1}}  
\newcommand\fractional[1]{\{#1\}}
\newcommand\ceil[1]{\lceil{#1}\rceil}
\newcommand\floor[1]{\lfloor{#1}\rfloor}
\newcommand{\Q}{{\mathbb Q}}
\renewcommand{\ll}{{\langle}}
\newcommand{\rr}{{\rangle}}

\newcommand{\CG}{{\mathcal{G}}}
\newcommand{\CE}{{\mathcal E}}

\newcommand{\CF}{{\mathcal F}}

\newcommand{\Res}{\operatorname{Res}}
\newcommand{\res}{\operatorname{res}}

\renewcommand{\ll}{{\langle}}

\newcommand\coneC{\mathfrak{c}}

\newcommand\coneU{\mathfrak{u}}

\DeclareMathOperator{\lcm}{lcm}
\renewcommand\d{\mathrm d}
\newcommand\e{\mathrm e}
\newcommand\T{\top}

\newcommand\vexi{\boldsymbol{\xi}}
\newcommand\vebeta{\boldsymbol{\beta}}
\newcommand\veeta{\boldsymbol{\eta}}

\newcommand\latteintegrale{{\tt LattE integrale}\xspace} 

\newcommand\maple{{\tt Maple}\xspace}
\newcommand\mapleKnapsack{\emph{M-Knapsack}\xspace}
\newcommand\latteKnapsack{\emph{LattE Knapsack}\xspace}
\newcommand\coneApx{\emph{LattE Top-Ehrhart}\xspace}

\newcommand\inlinefrac[2]{{#1}/{#2}}

\newcommand\maplecode[1]{\textsf{#1}}
\newcommand\shellcode[1]{\texttt{#1}}

\newcommand\vealpha{\boldsymbol{\alpha}}
\renewcommand{\a}{{\vealpha}}

\begin{document}
\maketitle
\begin{abstract}
For a given sequence $\a = [\alpha_1,\alpha_2,\dots,\alpha_{N+1}]$ 
of $N+1$ positive integers, we consider the combinatorial function 
$E(\a)(t)$ that counts the non-negative integer
solutions of the equation $\alpha_1x_1+\alpha_2 x_2+\cdots+\alpha_{N}
x_{N}+\alpha_{N+1}x_{N+1}=t$, where the right-hand side~$t$ is a
varying non-negative integer.  It is well-known that $E(\a)(t)$ is a
quasi-polynomial function in the variable $t$ of degree $N$. In
combinatorial number theory this function is known as Sylvester's
\emph{denumerant}.

Our main result is a new algorithm that, for every fixed number $k$,
computes in polynomial time the highest $k+1$ coefficients of the
quasi-polynomial $E(\a)(t)$ as  \emph{step polynomials} of~$t$ (a
simpler and more explicit representation).  Our algorithm
is a consequence of a nice poset structure on the poles of the
associated rational generating function for $E(\a)(t)$ and the
geometric reinterpretation of some rational generating functions in
terms of lattice points in polyhedral cones. Our
algorithm also uses  Barvinok's fundamental fast decomposition of
a polyhedral cone into unimodular cones. This paper also presents
a simple algorithm to predict the first \emph{non-constant} coefficient and concludes with a report of several 
computational experiments using an implementation of our algorithm in \latteintegrale. We compare it
with various \maple programs for partial or full computation of the denumerant.
\end{abstract}


\section{Introduction}

Let $\a=[\alpha_1,\alpha_2,\ldots, \alpha_{N},\alpha_{N+1}]$ be a
sequence of positive integers.  If $t$ is a non-negative integer, we
denote by $E(\a)(t)$ the number of solutions in non-negative integers
of the equation $\sum_{i=1}^{N+1} \alpha_i x_i=t$.  In other words,
$E(\a)(t)$ is the same as the number of partitions of the number $t$
using the parts $\alpha_1,\alpha_2,\ldots, \alpha_{N},\alpha_{N+1}$
(with repetitions allowed).  
 Let us begin with some background and history before
stating the precise results:

The combinatorial function $E(\a)(t)$ was called by J.~Sylvester the
\emph{denumerant}.  The denumerant $E(\a)(t)$ has a beautiful
structure: it has been known since the times of Cayley and
Sylvester that $E(\a)(t)$ is in fact a \emph{quasi-polynomial}, i.e.,
it can be written in the form $E(\a)(t)= \sum_{i=0}^{N} E_{i}(t)
t^{i}$, where $E_i(t)$ is a periodic function of $t$ (a more precise
description of the periods of the coefficients $E_i(t)$ will be given
later).  In other words, there exists a positive integer $Q$ such that
for $t$ in the coset $q+Q\Z$, the function $E(\a)(t)$ coincides with a
polynomial function of $t$.  This paper presents a new algorithm to
compute individual coefficients of this function and uncovers new
structure in generating functions that allows one to compute their
periodicity. The study of the coefficients $E_i(t)$,
in particular determining their periodicity, is a problem that has
occupied various authors and it is the key focus of our investigations
here.   
 Sylvester and Cayley first showed that the coefficients $E_i(t)$
are periodic functions having period equal to the least common 
 multiple of $\alpha_1,\ldots,\alpha_{N+1}$ 
(see
 \cite{beckgesselkomatsu,bellET} and references therein).  In 1943,
 E.\,T.~Bell gave a simpler proof and remarked that the period $Q$ is
 in the worst case given by the least common multiple of the $\alpha_i$,
 but in general it can be smaller. A classical observation that goes
 back to I.~Schur is that when the list $\a$ consist of relatively
 prime numbers, then asymptotically

$$ E(\a)(t) \approx \frac{t^N}{N!\, \alpha_1\alpha_2\cdots \alpha_{N+1}} \quad \text{as the number} \ t \rightarrow \infty. $$

Thus, in particular, there is a large enough integer $F$ such that for any $t \geq F$, $E(\a)(t)>0$ and
 there is a largest $t$ for which $E(\a)(t)=0$.
Let us give a simple example:

\begin{example}\label{example1}
Let $\a=[6,2,3].$  Then on each of the cosets $q+6\Z$, the  function
$E(\a)(t)$ coincides with a polynomial $E^{[q]}(t)$. Here are the
corresponding polynomials. 
\begin{align*}
E^{[0]}(t)&=\tfrac{1}{72}t^2+\tfrac{1}{4}t+1, &
E^{[1]}(t)&=\tfrac{1}{72}t^2+\tfrac{1}{18}t-\tfrac{5}{72}, \\
E^{[2]}(t)&=\tfrac{1}{72}t^2+\tfrac{7}{36}t+\tfrac{5}{9}, & 
E^{[3]}(t)&=\tfrac{1}{72}t^2+\tfrac{1}{6}t+\tfrac{3}{8}, \\
E^{[4]}(t)&=\tfrac{1}{72}t^2+\tfrac{5}{36}t+\tfrac{2}{9}, & 
E^{[5]}(t)&=\tfrac{1}{72}t^2+\tfrac{1}{9}t+\tfrac{7}{72}. 
\end{align*}
\end{example}

Naturally, the function $E(\a)(t)$ is equal to $0$ if $t$ does not
belong to the lattice $\sum_{i=1}^{N+1} \Z \alpha_i\subset \Z$ generated by
the integers $\alpha_i$. So if $g$ is the greatest common divisor of the
$\alpha_i$ (which can be computed in polynomial time), and
$\a/g=[\frac{\alpha_1}{g},\frac{\alpha_2}{g},\ldots, \frac{\alpha_{N+1}}{g}]$ the
formula $E(\a)(gt)=E(\a/g)(t)$ holds, and we may assume that the
numbers $\alpha_i$ span $\Z$ without changing the complexity of the problem.
In other words, we may assume that the greatest common divisor of the
$\alpha_i$ is equal to $1$. 

Our primary concern is how to compute $E(\a)(t)$, a problem has
received a lot of attention.  Computing the denumerant $E(\a)(t)$ as a
close formula or evaluating it for specific $t$ is relevant in several
other areas of mathematics.  In the combinatorics literature the
denumerant has been studied extensively (see e.g.,
\cite{agnarsson,beckgesselkomatsu, comtetbook,lisonek,riordanbook} and the
references therein).  The denumerant plays an important role in
integer optimization too \cite{kellereretalbook,martellotothbook}, where
the problem is called an \emph{equality-constrained knapsack}.  In
combinatorial number theory and the theory of partitions, the problem
appears in relation to the \emph{Frobenius problem} or the
\emph{coin-change problem} of finding the largest value of $t$ with
$E(\a)(t)=0$ (see \cite{wagonetal, kannanfrobenius,
  ramirezalfonsinbook} for details and algorithms).  Authors in the
theory of numerical semigroups have also investigated the so called
\emph{gaps} or \emph{holes} of the function (see \cite{hemmeckeetal}
and references therein), which are values of $t$ for which
$E(\a)(t)=0$, i.e., those positive integers $t$ which cannot be
represented by the $\alpha_i$.  For $N=1$ the number of gaps is
$(\alpha_1-1)(\alpha_2-1)/2$ but for larger $N$ the problem is quite
difficult.

Unfortunately, computing $E(\a)(t)$ or evaluating it are very
challenging computational problems. Even deciding whether $E(\a)(t)>0$
for a given $t$, is a well-known (weakly) NP-hard problem. Computing
$E(\a)(t)$, i.e., determining the number of solutions for a given~$t$,
is $\#P$-hard. Computing the Frobenius number is also known to
be NP-hard \cite{ramirezalfonsinbook}. Likewise, for a given
coset~$q+Q\Z$, computing the polynomial $E^{[q]}(t)$ is NP-hard.
Despite the difficulty to compute the function, in some special cases
one can compute information efficiently. For example, the Frobenius
number can be computed in polynomial time when $N+1$ is fixed
\cite{kannanfrobenius,barvinokwood}. At the same time for \emph{fixed}
$N+1$ one can compute the entire quasi-polynomial $E(\a)(t)$ in
polynomial time as a special case of a well-known result of Barvinok
\cite{bar}.  There are several papers exploring the practical
computation of the Frobenius numbers (see e.g., \cite{wagonetal} and
the many references therein).

We are certainly not the first to use generating functions to compute $E(\a)(t)$. Already Ehrhart obtained formulas for $E(\a)(t)$ in
terms of binomial coefficients using partial fraction decomposition. Similary, in \cite{sillszeilberger} 
the authors propose another way to recover the coefficients of the quasi-polynomial by a method they named \emph{rigorous guessing}.
In \cite{sillszeilberger} quasi-polynomials are represented as a function $f(t)$ given by $q$ polynomials $f^{[1]}(t), f^{[2]}(t), \dots,f^{[q]}(t)$ 
such that $f(t)=f^{[i]}(t)$ when $t \equiv i \pmod{q}$. To find the coefficients of the $f^{[i]}$ their method finds the 
first few terms of the Maclaurin expansion of the partial fraction decomposition to find enough evaluations of those polynomials
and then recovers the coefficients of the $f^{[i]}$ as a result of solving a linear system. Here we are able to prove good complexity results
and produced faster practical algorithms using the number-theoretic nature of the question.

It should be noted that  the polynomial-time complexity results for fixed 
$N$ were achieved using a powerful geometric interpretation of $E(\a)(t)$ (which was the original way we
encountered the problem too): The function $E(\a)(t)$ can also be
thought of as the number of integral points in the $N$-dimensional
simplex in $\R^{N+1}$ defined by
$\Delta_{\a}=\{\,[x_1,x_2,\ldots,x_N,x_{N+1}] : x_i\geq 0,
\sum_{i=1}^{N+1} \alpha_i x_i=t\,\}$ with rational vertices $\ve
s_i=[0,\ldots, 0,\frac{t}{\alpha_i},0,\ldots,0]$.  In this context,
$E(\a)(t)$ is a very special case of the \emph{Ehrhart function} (in
honor of French mathematician Eug\`ene Ehrhart who started its study
\cite{ehrhartbook}). Ehrhart functions count the lattice points inside
a convex polytope $P$ as it is dilated $t$ times. All of the results
we mentioned about $E(\a)(t)$ are in fact special cases of theorems
from Ehrhart theory \cite{barvinokzurichbook}. For example, the
asymptotic result of I.~Schur can be recovered from seeing that the
highest-degree coefficient of~$E_{\a}(t)$ is just the normalized
$N$-dimensional volume of the simplex~$\Delta_{\a}$. Our coefficients
are very special cases of Ehrhart coefficients.

This paper is about the computation of the coefficients of $E(\a)(t)$. Here are our main results:

\begin{enumerate}

\item  It is clear that the leading coefficient  is given by Schur's result. Our main result is a new algorithm for computing explicit formulas for more coefficients. 

\begin{theorem} \label{theo:complexity}
Given any fixed integer~$k$, there is a polynomial time algorithm to compute the highest $k+1$ degree
terms of the quasi-polynomial $E(\a)(t)$, that is
$$\mathrm{Top}_kE(\a)(t)=\sum_{i=0}^k E_{N-i}(t) t^{N-i}.$$
The coefficients are recovered as \emph{step polynomial} functions of $t$.
\end{theorem}

Note that the number~$Q$ of cosets for $E(\a)(t)$ can be exponential in the binary encoding size of the problem, and thus it is 
impossible to list, in polynomial time, the polynomials~$E^{[q]}(t)$ for all the cosets~$q+Q\Z$.  That is why to obtain a polynomial time algorithm, the output is presented in the format of \emph{step polynomials}, 
which we now introduce:  


  \begin{enumerate}[\rm(i)]
  \item 
    We first define the function $\fractional{s}=s-\floor{s} \in [0,1)$ 
    for $s\in\R$, where 
    $\floor{s}$ denotes the largest integer smaller or equal
    to~$s$. The function $\fractional{s+1}=\fractional{s}$ is a periodic
    function of $s$ modulo $1$.
  \item If $r$ is rational with denominator $q$, the function $T\mapsto
    \fractional{rT}$ is a function of $T\in \R$ periodic modulo~$q$.  A
    function of the form $T\mapsto \sum_i c_i \fractional{r_iT}$ will be called a
    \emph{(rational) step linear function}.  If all the $r_i$ have a common denominator $q$,
    this function is periodic modulo~$q$.
   
  \item Then consider the algebra generated over~$\Q$ by
    such functions on~$\R$. An element $\phi$ of this algebra can be written
    (not in a unique way) as
    $$ \phi(T) = \sum_{l=1}^L c_l \prod_{j=1}^{J_l} \fractional{r_{l,j} T}^{n_{l,j}}.$$
    Such a function $\phi(T)$ will be called a \emph{(rational) step polynomial}.
  \item We  will say that the step polynomial $\phi$ is of \emph{degree} (at most) $u$ if $\sum_j
    n_{l,j}\leq u$ for each index~$l$ occurring in the formula for
    $\phi$.\footnote{This notion of degree only induces a filtration, not a
      grading, on the algebra of step polynomials, because there exist
      polynomial relations between step linear functions and therefore several
      step-polynomial formulas with different degrees may represent the same
      function.} 
    We will say that $\phi$ is of \emph{period} $q$ if all the rational
    numbers $r_j$ have common denominator $q$.
  \end{enumerate}
 
In Example~\ref{example1}, instead of the $Q=6$ polynomials $E^{[0]}(t),
\dots, E^{[5]}(t)$ that we wrote down, we would write a single closed formula,
where the coefficients of powers of~$t$ are step polynomials in~$t$:
$$ {\frac {1}{72}}\,{t}^{2}+\left( \frac{1}{4}-\frac{\{-\frac{t}{3}\}}{6}-\frac{\{\frac{t}{2}\}}{6}\right) \, t + \left (1-\frac{3}{2}\, \{-\tfrac{t}{3}\}- \frac{3}{2}\,\fractional{\tfrac{t}{2}}+\frac{1}{2}\, \left( \fractional{-\tfrac{t}{3}} \right) ^{2}+ \fractional{-\tfrac{t}{3}}\fractional{\tfrac{t}{2}}+\frac{1}{2}\, \left( \fractional{\tfrac{t}{2}} \right) ^{2} \right).$$
For larger $Q$, one can see that this step polynomial representation is much more economical than writing
the individual polynomials for each of the cosets of the period~$Q$. 

Our results come after an earlier result of Barvinok \cite{barvinok-2006-ehrhart-quasipolynomial} who first proved
a similar theorem valid for all simplices. Also in \cite{so-called-paper-1}, the authors presented a polynomial-time algorithm 
 to compute the coefficient functions of  $\mathrm{Top}_kE(P)(t)$ for any simple polytope~$P$ (given by 
its rational vertices) in the form of \emph{step polynomials} defined as above. We note that both of these earlier papers use 
the geometry of the problem very strongly; instead our new algorithm is different as it uses more of the number-theoretic  structure of
 the special case at hand. There is a marked advantage of our algorithms over
 the work in \cite{barvinok-2006-ehrhart-quasipolynomial}:  We compute in a
 closed formula using the step polynomials all the possibilities of $E^{[q]}(t)$ while
 \cite{barvinok-2006-ehrhart-quasipolynomial} recovers a single
 polynomial~$E^{[q]}(t)$ for a given $q$. More importantly, our new algorithm is much easier to implement. Another relevant prior work (also useful
 for comparison) is our  algorithm  \coneApx presented in \cite{so-called-paper-1}. In that paper we extend Barvinok's results of 
\cite{barvinok-2006-ehrhart-quasipolynomial} to weighted Ehrhart quasi-polynomials via variation of his original approach.
The other important ingredient used in the efficient computation of the top coefficients is  the reinterpretation
of some generating functions in terms of lattice points in cones. This allows us to apply the polynomial-time 
signed cone decomposition of Barvinok for simplicial cones of fixed dimension~$k$~\cite{bar}.

\item Although the main result is computational, interesting
  mathematics comes into play: the new algorithm uses directly the
  residue theorem in one complex variable, which can be applied more
  efficiently as a consequence of a rich poset structure on the set of
  poles of the associated rational generating function for $E(\a)(t)$
  (see Subsection \ref{posetpoles}).  By Schur's result, it is clear that the  coefficient $E_N(t)$ of the highest degree term is just an explicit constant. Our analysis of the high-order
  poles of the generating function associated to $E(\a)(t)$ allows us
  to decide what is the highest-degree coefficient of $E(\a)(t)$ that
  is not a constant function of $t$ (we will also say that the coefficient is \emph{strictly periodic}).
\begin{theorem} \label{theo:firstperiodico}
Given a list of non-negative integer numbers $\a=[\alpha_1, \dots,
  \alpha_{N+1}]$, let $\ell$ be the greatest integer for which there exists a sublist $\a_J$ with $|J|=\ell$, such that its greatest common divisor is not $1$. Then for $k\geq \ell$ the coefficient of degree $k$ is a constant while the coefficient of degree $\ell-1$ of the quasi-polynomial $E(\a)(t)$  is strictly periodic. Moreover, if the numbers $\alpha_i$ are given with their prime factorization, then detecting $\ell$ can be done in polynomial time.

\end{theorem}

\begin{example} We apply the theorem above to investigate the 
question of periodicity of the denumerant coefficients
in the case of the classical partition problem
$E([1,2,3,\dots,m])(t)$. It is well known that this coincides with the
classical problem of finding the number of partitions of the integer
$t$ into at most $m$ parts, usually denoted $p_m(t)$ (see
\cite{andrewsbook}).  In this case, Theorem \ref{theo:firstperiodico}
predicts indeed that the highest-degree  coefficient
of the partition function $p_m(t)$ which is non-constant  is the coefficient of the term of degree $\lceil m/2
\rceil$.  This follows from the theorem because the even numbers in
the set $\{1,2,3,\dots,m\}$ form the largest sublist with gcd two. \end{example}

\item  The paper closes with an extensive collection of computational experiments (Section~\ref{experiments}). 
We constructed a dataset of over 760 knapsacks and show our new algorithm is the fastest available method for computing 
the top $k$ terms in the Ehrhart quasi-polynomial.  Our implementation of the
new algorithm is made available as a part of the free software
\latteintegrale \cite{latteintegrale}, version 1.7.2.\footnote{Available under the GNU General Public
  License at \url{https://www.math.ucdavis.edu/~latte/}.} 
\end{enumerate}


\section{The residue formula for $E(\a)(t)$}

Let us begin fixing some notation.
If $\phi(z)\,\d{z}$ is a meromorphic one form  on $\C$, with a pole at $z=\zeta$, we write
$$\Res_{z=\zeta}\phi(z)\,\d{z}=\frac{1}{2\pi i}\int_{C_\zeta} \phi(z) \,\d{z},$$ 
where $C_\zeta$ is a small circle around the pole $\zeta$.
 If $\phi(z)=\sum_{k\geq k_0} \phi_k z^k$ is a Laurent series in $z$, we denote by $\res_{z=0}$ the coefficient of $z^{-1}$ of $\phi(z)$.
  Cauchy's formula implies that
 $\res_{z=0} \phi(z)=\Res_{z=0} \phi(z)\,\d{z}$.

\subsection{\boldmath A residue formula for $E(\a)(t)$.}

Let $\a=[\alpha_1,\alpha_2,\ldots, \alpha_{N+1}]$  be a list of integers. Define
 $$F(\a)(z):=\frac{1}{\prod_{i=1}^{N+1}(1-z^{\alpha_i})}.$$

 Denote by  $\mathcal P=\bigcup_{i=1}^{N+1}\{\,\zeta\in \C: \zeta^{\alpha_i}=1\,\}$
 the set of poles of the meromorphic function $F(\a)$ and  by $p(\zeta)$ the order of
 the pole $\zeta$ for $\zeta\in \mathcal P$.


Note that because the $\alpha_i$ have greatest common divisor~$1$, we
have $\zeta=1$ as a pole of order ${N+1}$, and the other poles have
order strictly smaller.

\begin{theorem}\label{theo:sum}
Let $\a=[\alpha_1,\alpha_2,\ldots, \alpha_{N+1}]$  be a list of  integers with 
greatest common divisor equal to $1$, and let  $$F(\a)(z):=\frac{1}{\prod_{i=1}^{N+1}(1-z^{\alpha_i})}.$$
If $t$ is a non-negative integer, then
\begin{equation}
  E(\a)(t)=-\sum_{\zeta\in \mathcal P} \Res_{z=\zeta} z^{-t-1}F(\a)(z)\,\d{z}
  \label{eq:Ea-as-sum-of-residues}
\end{equation}
and the $\zeta$-term 
of this sum is a quasi-polynomial function of $t$ with degree less than or
equal to $p(\zeta)-1$. 

\end{theorem}

\begin{proof}
For $|z|<1$,
we write $\frac{1}{1-z^{\alpha_i}}=\sum_{u=0}^{\infty} z^{u\alpha_i}$ so that $F(\a)(z)=\sum_{t\geq 0} E(\a)(t)z^t.$

For a small circle $|z|=\epsilon$ of radius $\epsilon$ around $0$, the
integral of $z^k\,\d{z}$ is equal to $0$ except if $k=-1$, when it is $2\pi i$. 
Thus
$$E(\a)(t)=\frac{1}{2\pi i}\int_{|z|=\epsilon} z^{-t} F(\a)(z)
\frac{\d{z}}{z} = \frac{1}{2\pi i}\int_{|z|=\epsilon} z^{-t}
\prod_{i=1}^{{N+1}}\frac{1}{(1-z^{\alpha_i})} \frac{\d{z}}{z}.$$ Because
the $\alpha_i$ are positive integers, and $t$ a non-negative integer,
there are no residues at $z=\infty$ and we obtain
Equation~\eqref{eq:Ea-as-sum-of-residues} by applying the residue
theorem (for a reference about computational complex analysis see \cite{henrici3,henrici1,henrici2}.)


Write  $E_\zeta(t):=-\Res_{z=\zeta} z^{-t}F(\a)(z)\frac{\d{z}}{z}$; then the dependence in $t$ of $E_\zeta(t)$
 comes from the expansion of $z^{-t}$ near $z=\zeta$.
 We write $z=\zeta+y$, so that
$$E_\zeta(t)=-\Res_{y=0} (\zeta+y)^{-t}F(\a)(\zeta+y)\frac{\d
   y}{\zeta+y}.$$ As the pole of $F(\a)(\zeta+y)$ at $y=0$ is of order
 $p(\zeta)$, to compute the residue at $y=0$, we only need to expand
 in $y$ the function $(\zeta+y)^{-t-1}$ and take the coefficient of
 $y^{p(\zeta)-1}.$ Now from the generalized Newton binomial theorem,
 for $k=t+1$ the function $(\zeta+y)^{-k}=\sum_{n=0}^\infty \binom{n+k-1}{n}
 \zeta^{-k-n}(-y)^n$. From this expression one can recover the desired
 coefficient.

One can easily check that the dependence in $t$ of our residue is a
quasi-polynomial with degree less than or equal to $p(\zeta)-1$.  We
thus obtain the result.
\end{proof}

\subsection{Poles of high and low order}

Given an integer $0\leq k\leq N$,
we partition  the set of poles $\mathcal P$ in two disjoint sets according to
the order of the pole:
$$\mathcal P_{> N-k}=\{\,\zeta: p(\zeta)\geq N+1-k\,\},\qquad
\mathcal P_{\leq N-k} =\{\,\zeta: p(\zeta)\leq {N}-k\,\}.$$


\begin{example}\label{ex:poles}
  \begin{enumerate}[\rm(a)]
  \item  Let $\a= [98, 59, 44, 100]$, so $N=3$, and let $k=1$.
    Then $\mathcal P_{> N-k}$ consists of poles of order greater than~$2$. Of
    course $\zeta=1$ is a pole of order $4$.  Note that $\zeta=-1$ is a pole of
    order $3$.  So $\mathcal P_{> N-k}=\{\,\zeta : \zeta^2=1\,\}$.
  \item Let $\a= [6,2,2,3,3]$, so $N=4$, and let $k=2$.
    Let $\zeta_6 = \e^{2\pi i/6}$ be a primitive 6th root of unity. 
    Then $\zeta_6^6 = 1$ is a pole of order~$5$, $\zeta_6$ and $\zeta_6^5$ are poles of
    order~1, and $\zeta_6^2$, $\zeta_6^3 = -1$, $\zeta_6^4$ are poles of
    order~3. Thus $\mathcal P_{> N-k}=\mathcal P_{> 2}$ is the union of $\{\,\zeta :
    \zeta^2=1\,\} = \{-1,1\} $ and $\{\,\zeta : \zeta^3=1\,\} = \{
    \zeta_6^2, \zeta_6^4, \zeta_6^6 = 1\}$.
  \end{enumerate}
\end{example}

According to the disjoint decomposition $\mathcal P=\mathcal P_{\leq N-k}\cup \mathcal P_{> N-k}$, we write
\begin{align*}
E_{\mathcal P_{> N-k}}(t)&=-\sum_{\zeta\in \mathcal P_{> N-k}} \Res_{z=\zeta}
z^{-t-1}F(\a)(z)\,\d{z} \\
\intertext{and}  
E_{\mathcal P_{\leq N-k}}(t)&=-\sum_{\zeta\in \mathcal P_{\leq N-k}} \Res_{z=\zeta} z^{-t-1}F(\a)(z)\,\d{z}.
\end{align*}


The following proposition is a direct consequence of Theorem \ref{theo:sum}.
 \begin{proposition}
We have
$$E(\a)(t)=E_{\mathcal P_{> N-k}}(t)+E_{\mathcal P_{\leq N-k}}(t),$$
where the function  $E_{\mathcal P_{\leq N-k}}(t)$ is a quasi-polynomial function in the variable $t$ of degree  strictly less than $N-k$.
\end{proposition}

Thus  for the purpose of computing ${\rm Top}_kE(\a)(t)$ it is sufficient to compute the function
$E_{\mathcal P_{> N-k}}(t)$.  This function is computable in polynomial
time, as stated in the main result of our paper:
\begin{theorem}\label{theo:complexity}
  Let $k$ be a fixed number.  Then the coefficient functions of the quasi-polynomial function
  $E_{\mathcal P_{> N-k}}(t)$ are computable in polynomial time as step
  polynomials of~$t$.
\end{theorem}

We prove the theorem in the rest of this section and the next.

\subsection{The poset of the high-order poles} \label{posetpoles}

We first rewrite our set ${\mathcal P_{> N-k}}$.  Note
that if $\zeta$ is a pole of order $\geq p$, this means that there exist at
least $p$ elements $\alpha_i$ in the list $\a$ so that $\zeta^{\alpha_i}=1$. But if
$\zeta^{\alpha_i}=1$ for a set~$I\subseteq\{1,\dots,N+1\}$ of indices $i$, this is
equivalent to the fact that $\zeta^f=1$, for $f$ the greatest common divisor
of the elements $\alpha_i, i\in I$.

Now let $\mathcal I_{> N-k}$ be the set of subsets of $\{1, \dots, N+1\}$ of cardinality greater than $N-k$. Note that when $k$ is fixed,
the cardinality of $\mathcal I_{> N-k}$ is a polynomial function of~$N$.
For each subset $I\in
\mathcal I_{> N-k}$, define $f_I$ to be the greatest common divisor of the corresponding sublist $\alpha_i$, $i\in I$.  Let $\CG_{>N-k}(\a)=\{\,f_I : I \in \mathcal I_{>
  N-k}\,\}$ be the set of integers so obtained and let $G(f)\subset\C^\times$ be the group of $f$-th roots of unity, 
$$G(f)=\{\,\zeta\in \C: \zeta^f=1\,\}.$$ The set $\{\, G(f) : f \in
\CG_{>N-k}(\a)\,\}$ forms a poset $\tilde{P}_{>N-k}$ (partially ordered set)
with respect to reverse inclusion.  That is, $G(f_i) \preceq_{\tilde{P}_{>N-k}} G(f_j)$ if $G(f_j) \subseteq G(f_i)$ (the $i$ and $j$ become swapped). Notice $G(f_j) \subseteq G(f_i) \Leftrightarrow f_j $ divides $f_i$. Even if $\tilde{P}_{>N-k}$ has a unique minimal element, we add an element $\hat{0}$ such that $\hat{0} \preceq G(f)$ and call this new poset $P_{>N-k}$.

In terms of the group $G(f)$ we have thus $\mathcal P_{>
  N-k}=\bigcup_{f\in \CG_{>N-k}(\a)} G(f)$. This is, of course, not a
disjoint union, but using the inclusion--exclusion principle, we can write
the indicator function of the set  $\mathcal P_{> N-k}$ as a linear combination
of indicator functions of the sets $G(f)$:
$$[\mathcal P_{> N-k}]=\sum_{f\in \CG_{>N-k}(\a)} \mu_{>N-k}(f) [G(f)],$$ where $\mu_{>N-k}(f) := -\mu'_{>N-k}(\hat{0},G(f))$ and $\mu'_{>N-k}(x,y)$ is the standard M\"obius function for the poset $P_{>N-k}$:
\begin{align*}
\mu'_{>N-k}(s,s) &= 1 && \forall s \in P_{>N-k}, \\
\mu'_{>N-k}(s,u) &= -\sum\limits_{s \preceq t \prec u} \mu'_{>N-k}(s,t) && \forall s \prec u\text{ in } P_{>N-k}.
\end{align*}
For simplicity,  $\mu_{>N-k}$ will be called the \emph{M\"obius function} for the poset $P_{>N-k}$ and will be denoted simply by $\mu(f)$. We also have the relationship
\begin{align*}
\mu(f) &= -\mu'_{>N-k}(\hat{0},G(f)) \\
&= 1 + \sum\limits_{\hat{0} \prec G(t) \prec G(f)} \mu'_{>N-k}(\hat{0},G(t))\\
&= 1 - \sum\limits_{\hat{0} \prec G(t) \prec G(f)} -\mu'_{>N-k}(\hat{0},G(t))\\
&= 1 - \sum\limits_{\hat{0} \prec G(t) \prec G(f)} \mu(t).
\end{align*}
\begin{example}[Example~\ref{ex:poles}, continued]~
  \begin{enumerate}[\rm(a)]
  \item Here we have $\mathcal I_{>N-k} = \mathcal I_{>2} = \bigl\{
    \{\oldstylenums1, \oldstylenums2, \oldstylenums3\},
    \{\oldstylenums1,\oldstylenums2,\oldstylenums4\},
    \{\oldstylenums1,\oldstylenums3,\oldstylenums4\},
    \{\oldstylenums2,\oldstylenums3,\oldstylenums4\},\allowbreak
    \{\oldstylenums1,\oldstylenums2,\oldstylenums3,\oldstylenums4\} \bigr\}$ 
    and $\mathcal G_{>N-k}(\a) = \{ 1, 1, 2, 1, 1\} = \{1,2\}$.  Accordingly, 
    $\mathcal P_{> N-k} = G(1) \cup G(2)$.  The poset $P_{>2}$ is
    
    \begin{figure}[ht]
\begin{tikzpicture}[scale=0.50,transform shape,->,>=stealth',shorten >=1pt,auto,node distance=3cm,
	  thick,main node/.style={circle,draw,font=\sffamily\Large\bfseries}]
	  \node[main node] (1) {$G(1)$};
	  \node[main node] (2) [below of =1] {$G(2)$};
	  \node[main node] (3) [below  of =2] {$\hat{0}$};
	
	  \path[every node/.style={font=\sffamily\small}]
		 (1) edge node  {} (2)
		(2) edge node {} (3);
\end{tikzpicture}
\end{figure}

    The arrows denote subsets, that is $G(1) \subset G(2)$ and $\hat{0}$ can be identified with the unit circle. The M\"obius function~$\mu$ is simply given by $\mu(1) = 0$, $\mu(2) = 1$, and so
    $[\mathcal P_{>N-k}] = [G(2)]$.
  \item Now $\mathcal I_{>N-k} = \mathcal I_{>2} = \bigl\{
    \{\oldstylenums1, \oldstylenums2, \oldstylenums3\},
    \{\oldstylenums1,\oldstylenums2,\oldstylenums4\},\dots,\allowbreak
    \{\oldstylenums3,\oldstylenums4,\oldstylenums5\},\allowbreak
    \{\oldstylenums1, \oldstylenums2, \oldstylenums3, \oldstylenums4\},\allowbreak
    \{\oldstylenums1, \oldstylenums2, \oldstylenums3, \oldstylenums5\},\allowbreak
    \{\oldstylenums1,\oldstylenums2,\oldstylenums4, \oldstylenums5\},\allowbreak
    \{\oldstylenums1,\oldstylenums3,\oldstylenums4, \oldstylenums5\},\allowbreak
    \{\oldstylenums2,\oldstylenums3,\oldstylenums4, \oldstylenums5\},\allowbreak
    \{\oldstylenums1,\oldstylenums2,\oldstylenums3,\oldstylenums4,\oldstylenums5\}
    \bigr\}$  
    and thus $\mathcal G_{>N-k}(\a) = \{ 2, 3, 1, 1 \} \allowbreak =
    \allowbreak \{1,2,3\}$. Hence
    $\mathcal P_{> N-k} = G(1) \cup G(2)\cup G(3) = \{1\} \cup \{-1,1\} \cup
    \{\zeta_3, \zeta_3^2, 1\}$, where $\zeta_3 = \e^{2\pi i/3}$ is a primitive
    3rd root of unity. 
    
    \begin{figure}[ht]
\begin{tikzpicture}[scale=0.50,transform shape,->,>=stealth',shorten >=1pt,auto,node distance=3cm,
	  thick,main node/.style={circle,draw,font=\sffamily\Large\bfseries}]
	  \node[main node] (1) {$G(1)$};
	  \node[main node] (2) [below left of =1] {$G(2)$};
	  \node[main node] (3) [below right of =1] {$G(3)$};
	  \node[main node] (4) [below  right of =2] {$\hat{0}$};
	
	  \path[every node/.style={font=\sffamily\small}]
		 (1) edge node  {} (2)
		      edge node {} (3)
		(2) edge node {} (4)
		(3) edge node {} (4);
\end{tikzpicture}
\end{figure}    
    
    The M\"obius function~$\mu$ is then $\mu(3)=1$, $\mu(2)=1$,
    $\mu(1)=-1$, and thus 
    $[\mathcal P_{>N-k}] = -[G(1)] + [G(2)] + [G(3)]$.
  \end{enumerate}
\end{example}


\begin{theorem}
Given a list $\a=[\alpha_1, \dots, \alpha_{N+1}]$ and a fixed integer $k$, then the values for the M\"obius function for the poset $P_{>N-k}$ can be computed in polynomial time.
\end{theorem}
\begin{proof}
First find the greatest common divisor  of all sublists of  the list $\a$ with size greater than $N-k$. Let $V$ be the set of integers obtained from all such greatest common divisors. We note that each node of the poset $P_{>N-k}$ is a group of roots of unity $G(v)$. But it
is labeled by a non-negative integer $v$.

Construct an array $M$ of size $|V|$ to keep the value of the M\"obius function. Initialize $M$ to hold the M\"obius values of infinity: $M[v] \leftarrow \infty$ for all $v \in V$. Then call Algorithm \ref{alg:findMobius} below with $\text{findM\"obius}(1,V,M)$.

\begin{algorithm}   \label{alg:findMobius}
\caption{ findM\"obius($n$, $V$, $M$)}

\begin{algorithmic}[1]                    
\REQUIRE $n$: the label of node $G(n)$ in the poset $\tilde{P}_{>N-k}$
\REQUIRE $V$: list of numbers in the poset $\tilde{P}_{>N-k}$
\REQUIRE $M$: array of current M\"obius values computed for $P_{>N-k}$
\ENSURE updates the array $M$ of M\"obius values
\IF{$M[n] < \infty$}
	\RETURN
\ENDIF
\STATE $L \leftarrow \{\, v \in V  : n \mid v\,\} \setminus \{n\}$
\IF{ $L = \emptyset $}
	\STATE $M[n] \leftarrow 1$
	\RETURN
\ENDIF
\STATE $M[n] \leftarrow 0$
\FORALL {$v \in L$}
	\STATE findM\"obius($v, L, M$)
	\STATE $M[n] \leftarrow M[n] + M[v]$
\ENDFOR
\STATE $M[n] \leftarrow 1- M[n]$

\end{algorithmic}
\end{algorithm}

Algorithm \ref{alg:findMobius} terminates because the number of nodes $v$ with $M[v] = \infty$ decreases to zero in each iteration. 
To show correctness, consider a node $v$ in the poset $P_{N-k}$. If $v$ covers $\hat{0}$, then we must have $M[v] = 1$ as there is no other $G(w)$ with  $G(f) \subset G(w)$. Else if $v$ does not cover $\hat{0}$, we set $M[v]$ to be 1 minus the sum $\sum\limits_{w :\; v \mid w} M[w]$ which guarantees that the poles in $G(v)$ are only counted once because $\sum\limits_{w :\; v \mid w} M[w]$  is how many times $G(v)$ is a subset of another element that has already been counted.

The number of sublists of $\a$ considered is $\binom{N}{1}+\binom{N}{2}+\cdots+\binom{N}{k} = O(N^k)$, which is a polynomial for $k$ fixed. For each sublist, the greatest common divisor of a set of integers is computed in polynomial time. Hence $|V| = O(N^k)$. Notice that lines $4$ to $14$ of Algorithm \ref{alg:findMobius} are executed at most $O(|V|)$ times as once a $M[v]$ value is computed, it is never recomputed. The number of additions on line $12$ is $O(|V|^2)$ while the number of divisions on line $4$ is also $O(|V|^2)$. Hence this algorithm finds the M\"obius function in $O(|V|^2) = O(N^{2k})$ time where $k$ is fixed.
\end{proof}

Let us define  for  any positive integer $f$  $$E(\a,f)(t)=-\sum_{\zeta:\  \zeta^f=1}  \Res_{z=\zeta} z^{-t-1} F(\a)(z)\,\d{z}.$$

\begin{proposition}\label{prop}
  Let $k$ be a fixed integer, then 
\begin{equation}\label{eq:Emu}
   E_{\mathcal P_{> N-k}}(t)=-\sum_{f\in \CG_{>N-k}(\a)}
     \mu(f)  E(\a,f)(t).
   \end{equation}
\end{proposition}

Thus we have reduced the computation to the fast computation of $E(\a,f)(t)$. 


\section{Polyhedral  reinterpretation of the generating function $E(\a,f)(t)$}

To complete the proof of Theorem \ref {theo:complexity} we need only to prove the following proposition.

\begin{proposition}
   For any integer $f\in \CG_{>N-k}(\a)$, the coefficient functions of the
     quasi-polynomial function $E(\a,f)(t)$ and hence $E_{\mathcal P_{>
         N-k}}(t)$ are computed in polynomial time as step polynomials of~$t$.
\end{proposition}

By Proposition \ref{prop} we know we need to compute the value of $E(\a,f)(t)$. Our goal now is to 
demonstrate that this function can be thought of as the generating function of the lattice points inside
a convex cone. This is a key point to guarantee good computational bounds. Before we can do that
we review some preliminaries on generating functions of cones. We recall the notion of generating 
functions of cones; see also \cite{so-called-paper-1}.

Let $V=\R^r$ provided with a lattice $\Lambda$, and let $V^*$ denote the dual
space. A \emph{(rational) simplicial cone}  $\coneC{} =  \R_{\geq0}\ve w_1+\dots+\R_{\geq0}\ve w_r$ is
a cone generated by $r$ linearly independent vectors~$\ve w_1,\dots,\ve w_r$ of~$\Lambda$.
We consider the semi-rational affine cone $\ve s+\coneC{}$, $\ve s\in V$. 
Let $\vexi\in V^*$ be a dual vector such that $\ll\vexi,\ve w_i\rr <0, \ 1\leq i\leq r.$ Then the sum
$$S(\ve s+\coneC{},\Lambda)(\vexi)=\sum_{\ve n \in   (\ve s+\coneC{})\cap\Lambda}
\e^{\langle \vexi,\ve n\rangle}$$ is summable and defines an analytic function of $\vexi$.  It is well known that
this function extends to a meromorphic function of $\vexi\in V^*_\C$. We still
denote this meromorphic extension by $S(\ve s+\coneC{},\Lambda)(\vexi)$.

\begin{example}\label{ex:dim1}
  Let $V=\R$  with lattice $\Z$,  $\coneC{}=\R_{\geq 0}$, and $s\in  \R$.
  Then $$S(s+\R_{\geq0},\Z)(\xi)=\sum_{n\geq s} \e^{n \xi}=\e^{\ceil{s}\xi}\frac{1}{1-\e^{\xi}}.$$
  Using the function 
  $\fractional{x}=x - \floor{x}$, we find $\ceil{s} = s + \fractional{-s}$ and
  can write
  \begin{equation}\label{eq:dim1}
    \e^{-s\xi}S(s+\R_{\geq0},\Z)(\xi)=\frac{\e^{\fractional{-s}\xi}}{1-\e^{\xi}}. 
  \end{equation}
\end{example}

Recall the following result:

\begin{theorem}
 Consider the semi-rational affine cone $\ve s+\coneC{}$ and the lattice $\Lambda$. The series $S(\ve s+\coneC{},\Lambda)(\vexi)$ is a meromorphic function of $\vexi$  such that $\prod_{i=1}^r \ll \vexi,\ve w_i\rr  \cdot\allowbreak S(\ve s+\coneC{},\Lambda)(\vexi)$ is
  holomorphic in a neighborhood of $\ve0$.
\end{theorem}

Let $\ve t\in \Lambda$. Consider the translated cone $\ve t+\ve
s+\coneC{}$ of
$\ve s+\coneC{}$ by $\ve t$. Then we have the covariance formula
\begin{equation}\label{eq:covariance}
  S(\ve t+\ve s+\coneC{},\Lambda)(\vexi)=\e^{\ll\vexi,\ve t\rr }  S(\ve s+\coneC{},\Lambda)(\vexi).
\end{equation}


Because of this formula, 
it is convenient to introduce the following function.
\begin{definition} \label{def:useful} 
  Define the function $$M(\ve s,\coneC{},\Lambda)(\vexi):=\e^{-\ll \vexi, \ve s\rr }
  S(\ve s+\coneC{},\Lambda)(\vexi).$$
\end{definition}
Thus the function $\ve s\mapsto M(\ve s,\coneC{},\Lambda)(\vexi)$ is a
function of $\ve s\in V/\Lambda$ (a periodic function of $\ve s$) whose values are
meromorphic functions of $\vexi$.
It is interesting to introduce this modified function since, as seen in
Equation \eqref{eq:dim1} in Example~\ref{ex:dim1}, its dependance in  $\ve s$
is via step linear functions of $\ve s.$ 

There is a very special and important case when the function $M(\ve s,\coneC{},\Lambda)(\vexi)=\e^{-\ll \vexi, \ve s\rr }
  S(\ve s+\coneC{},\Lambda)(\vexi)$ is easy to write down. A \emph{unimodular} cone, is a
cone $\coneU$ whose primitive generators $\ve g_i^\coneU$ form a basis of the
lattice $\Lambda$.  We introduce the following notation.
\begin{definition}
  Let $\coneU{}$ be a unimodular cone with primitive generators $\ve g_i^\coneU$ and
  let $\ve s\in V$. 
  Then, write  $\ve s=\sum_i s_i \ve g_i^\coneU$, with $s_i\in \R$, and         define
  $$\smallstep{\ve s}_\coneU{}=\sum_i \fractional{-s_i} \ve g_i^\coneU.$$
\end{definition}
Thus 
$\ve s+\smallstep{\ve s}_\coneU{}=\sum_i \ceil{s_i} \ve g_i^\coneU$.  Note that if $\ve t\in \Lambda$, 
then $\smallstep{(\ve s+\ve t)}_\coneU{}=\smallstep{\ve s}_\coneU{}$. Thus,
$\ve s\mapsto
\smallstep{\ve s}_\coneU{}$ is a function on $V/\Lambda$  with value in $V$. 
For any $\vexi\in V^*$, we then find
\begin{equation*}
S(\ve s+\coneU,\Lambda)(\vexi)= \e^{\ll \vexi,\ve
  s\rr}\e^{\ll\vexi,\smallstep{\ve s}_\coneU\rr}\frac{1}{\prod_j(1-\e^{\ll \vexi,\ve g_j^\coneU\rr})}
\end{equation*}
and thus
\begin{equation}\label{eq:M-uni}
M(\ve s,\coneU,\Lambda)(\vexi)= \e^{\ll\vexi,\smallstep{\ve
    s}_\coneU\rr}\frac{1}{\prod_j(1-\e^{\ll \vexi,\ve g_j^\coneU\rr})}. 
\end{equation}

For a general cone $\coneC{}$,
we can decompose its indicator function $[\coneC{}]$ as a signed sum of
indicator functions of unimodular cones, 
$\sum_\coneU \epsilon_\coneU [\coneU]$, modulo indicator functions of
cones containing lines. As shown 
by Barvinok (see \cite{bar} for the original source and \cite{barvinokzurichbook} for a great new exposition),  
if the dimension~$r$ of~$V$ is fixed, this decomposition can be computed in polynomial time.
Then we can write 
$$S(\ve s+\coneC{},\Lambda)(\vexi)=\sum_\coneU \epsilon_\coneU\, S(\ve s+\coneU,\Lambda)(\vexi).$$
Thus we obtain, using Formula (\ref{eq:M-uni}), 
\begin{equation}\label{formula:M}
M(\ve s, \coneC{},\Lambda)(\vexi)=\sum_{\coneU} \epsilon_\coneU\, \e^{ \ll
  \vexi,\smallstep{\ve s}_\coneU\rr } \frac{1}{\prod_j (1-\e^{\ll \vexi,\ve g_j^\coneU\rr })}.
\end{equation}
Here $\coneU$ runs through all the unimodular cones occurring in the decomposition of $\coneC{}$, and
the $\ve g_j^\coneU\in \Lambda$ are the corresponding generators of  the unimodular cone $\coneU.$

\begin{remark}\label{rem:change-to-standard-lattice}
For computing explicit examples, it is convenient to make a change of variables that
leads to computations in the standard lattice~$\Z^r$.  Let $B$ be the matrix
whose columns are the generators of the lattice~$\Lambda$; then
$\Lambda=B\Z^r$.  
\begin{align*}
  M(\ve s,\coneC{},\Lambda)(\vexi)
  &= \e^{-\ll \vexi, \ve s\rr } \sum_{\ve n \in (\ve s+\coneC{})\cap B\Z^r}
  \e^{\langle \vexi,\ve n\rangle} \\
  &= \e^{-\ll B^\T \vexi, B^{-1} \ve s\rr} \sum_{\ve x \in ( B^{-1}(\ve
    s+\coneC{})\cap\Z^r} \e^{\langle B^\T\vexi,\ve x\rangle} 
  = M( B^{-1}\ve s,B^{-1} \coneC{}, \Z^r)(B^\T\vexi).
\end{align*}
\end{remark}

\subsection{\boldmath Back to the computation of $E(\a,f)(t)$}

After the preliminaries we will see how to rewrite $E(\a,f)(t)$ in terms of lattice points of simplicial cones. This will
require some suitable manipulation of the initial form of $E(\a,f)(t)$. 
To start with, define the  function
$$\CE(\a,f)(t,T)=-\res_{x=0} \e^{-tx} \sum_{\zeta:\  \zeta^f=1} \frac{\zeta^{-T}}{\prod_{i=1}^{N+1} (1-\zeta^{\alpha_i}\e^{\alpha_i x})}.$$
Writing $z=\zeta \e^x,$  
changing coordinates in residue and computing $\d{z}=z\, \d{x}$  we write:
$$\CE(\a,f)(t,T)=-\res_{z=\zeta} z^{-t-1} \zeta^{t}\sum_{\zeta \colon   \zeta^f=1} \frac{\zeta^{-T}}{\prod_{i=1}^{N+1} (1-z^{\alpha_i})}.$$
By evaluating  at $T=t,$ we obtain:
\begin{equation}\label{eval}
E(\a,f)(t)=\CE(\a,f)(t,T) \big|_{T=t}.
\end{equation}
We can now define:

\begin{definition}\label{FE} Let $k$ be fixed.
For $f\in \mathcal G_{>N-k}(\a)$,   define
$$\CF(\a,f,T)(x):=\sum_{\zeta:\  \zeta^f=1} \frac{\zeta^{-T}}{\prod_{i=1}^{N+1} (1-\zeta^{\alpha_i}\e^{\alpha_i x})},$$
and 
$$E_i(f)(T):=\res_{x=0}\frac{(- x)^i}{i!} \CF(\a,f,T)(x).$$
\end{definition}
Then $$\CE(\a,f)(t,T)=-\res_{x=0} \e^{-tx} \CF(\a,f,T)(x).$$

The  dependence in $T$ of $\CF(\a,f,T)(x)$ is through $\zeta^T$. As
$\zeta^f=1$,  the function  $\CF(\a,f,T)(x)$ is a periodic function of $T$
modulo $f$ whose values are meromorphic functions of $x$.  Since the pole in $x$ is of order at most $N+1$, we can rewrite $\CE(\a,f)(t,T)$ in terms of  $E_i(f)(T)$ and prove:


%
%
\begin{theorem}\label{E} Let $k$ be fixed.
Then for  $f\in \mathcal G_{>N-k}(\a)$  we  can write
$$\CE(\a,f)(t,T)=\sum_{i=0}^N  t^i E_i(f)(T)$$ with $E_i(f)(T)$  a step
polynomial  of degree  less than or equal to $N-i$ and  periodic  of $T$ modulo
$f$.  This step polynomial can be computed in polynomial time.
\end{theorem}

It is now clear that once we have proved Theorem \ref{E}, then the proof of Theorem \ref {theo:complexity} will follow. Writing everything out, for  $m$ such that $0\leq m\leq N$, the coefficient of  $t^{m}$ in the Ehrhart quasi-polynomial is given by
\begin{equation}\label{eq:Ehrhartcoeff}
 E_m(T)= - \res_{x=0}\frac{(-x)^{m}}{m!}\sum_{f\in \CG_{>m}(\a) } \mu(f)\sum_{\zeta: \  \zeta^f=1}\frac{\zeta^{-T}}{\prod_i(1-\zeta^{\alpha_i}\e^{\alpha_i x})}.
\end{equation}
As an example, we see that $E_N$ is indeed independent of $T$ because
$\CG_{>N}(\a) = \{1\}$; thus  $E_N$ is a constant. 
We now concentrate on writing the function 
$\CF(\a,f,T)(x)$ more explicitly.

\begin{definition}\label{def:H}
For a list $\a$ and integers $f$ and $T$, define   meromorphic functions of $x\in \C$  by:
$${\mathcal B}(\a,f)(x):=\frac{1}{\prod_{i\colon f\mid \alpha_i}(1-\e^{\alpha_i x})},$$
$${\mathcal S}(\a,f,T)(x):=\sum_{\zeta:  \  \zeta^f=1} \frac{\zeta^{-T}}{\prod_{i\colon f\nmid \alpha_i} (1-\zeta^{\alpha_i}\e^{\alpha_i x})}.$$
\end{definition}
Thus 
$$\CF(\a,f,T)(x)={\mathcal B}(\a,f)(x)\, {\mathcal S}(\a,f,T)(x).$$

The expression we obtained will allow us to compute $\CF(\a,f,T)$ by relating
${\mathcal S}(\a,f,T)$ to a generating function of a cone.  This cone will have fixed
dimension when $k$ is fixed.

\subsection{\boldmath $E(\a,f)(t)$ as the generating function of a cone in fixed dimension}

To this end, let $f$ be an integer from $\mathcal G_{>N-k}(\a)$.  By
definition, $f$ is the greatest common divisor of a sublist of $\a$.  Thus
the greatest common divisor of $f$ and the elements of $\a$ which are
\emph{not} a multiple of $f$ is still equal to~$1$.
Let $J=J(\a,f)$ be the set of indices $i\in\{1,\dots,N+1\}$ such that $\alpha_i$ is
indivisible by~$f$, i.e., $f \nmid \alpha_i$.  Note that 
$f$ by definition is the greatest common divisor of all except at most $k$
of the integers $\alpha_j$.  Let $r$
denote the cardinality of~$J$; then $r\leq k$. Let $V_J=\R^J$ and let $V_J^*$ denote
the dual space. We will use the standard basis of $\R^J,$ and we denote by $\R^J_{\geq 0}$ the standard cone of elements in $\R^J$ having non-negative coordinates.  We also define the sublist $\a_J = [\alpha_i]_{i\in J}$ of elements
of~$\a$ indivisible by~$f$ and view it as a vector in $V_J^*$ via the standard basis. 
\begin{definition}
  For an integer $T$, define the meromorphic function of
  $\vexi\in V_J^*$, 
  $$Q({\a},f,T)(\vexi):=\sum_{\zeta:\  \zeta^{f}=1}
  \frac{\zeta^{-T}}{\prod_{j\in J(\a,f)} (1-\zeta^{\alpha_j}\e^{\xi_j})}.$$
\end{definition}
\begin{remark}
  \label{rem:restrict-to-SafT}
  Observe that $Q({\a},f,T)$ can be restricted at $\vexi=\a_J x$,
  for $x\in\C$ generic, to give $\mathcal S(\a,f,T)(x).$
\end{remark}


We  find that $Q({\a},f,T)(\vexi)$ is the discrete generating function of  an affine shift of the standard cone $\R^J_{\geq 0}$ relative to a certain lattice in $V_J$
which we define as:
\begin{equation}\label{lattice}
  \Lambda({\a},f):=\biggl\{\,\ve y \in \Z^J : \langle \a_J, \ve y\rangle = \sum_{j\in J} y_j\alpha_j\in\Z f\,\biggr\}. 
\end{equation}
Consider the map $\phi\colon \Z^J \to \Z/\Z f$, $\ve y\mapsto \langle \a, \ve
y\rangle + \Z f$.  Its kernel is the lattice~$\Lambda(\a,f)$. Because the
greatest common divisor of $f$ and the elements of $\a_J$ is~$1$, by Bezout's
theorem there exist $s_0\in\Z$ and $\ve s\in\Z^J$ such that $1=\sum_{i\in J}
s_i \alpha_i+s_0 f$.  Therefore, the map~$\phi$ is surjective, and therefore the
index $|\Z^J:\Lambda(\a,f)|$ equals~$f$.

\begin{theorem}\label{th:as-lattice-genfun}
  Let $\a=[\alpha_1,\dots,\alpha_{N+1}]$ be a list of positive integers and $f$ be the
  greatest common divisor of a sublist of~$\a$.  
  Let $J=J(\a,f) = \{\, i
  : f \nmid \alpha_i \,\}.$ 
  Let $s_0\in\Z$ and $\ve s\in\Z^J$ such that $1=\sum_{i\in J} s_i
  \alpha_i+s_0 f$ using Bezout's theorem.
  Consider $\ve s=(s_i)_{i\in J}$ as an element of $V_J = \R^J.$ Let $T$ be an integer,  and  $\vexi=(\xi_i)_{i\in J} \in V_J^*$ with  $ \xi_i <0.$ Then $$Q({\a},f,T)(\vexi)=f\, \e^{\ll \vexi,T \ve s\rr }\sum_{\ve n \in   (-T\ve s+ \R_{\geq 0}^J)\cap\Lambda(\a,f)}\e^{\langle \vexi,\ve n\rangle}$$\end{theorem}

\begin{remark}
The function  $Q(\a,f,T)(\vexi)$ is a function of $T$ periodic modulo $f$.
Since $f \Z^J$ is contained in $\Lambda(\a,f)$, the element $f \ve s$ is in
the lattice  $\Lambda(\a,f)$, and we see that the right hand side is also a
periodic function of $T$  modulo $f$. 
\end{remark}

\begin{proof}[Proof of Theorem~\ref{th:as-lattice-genfun}]
  Consider $\vexi\in V_J^*$ with $\xi_j<0$. 
  Then we can write the equality
  $$\frac{1} {\prod_{j\in J} (1-\zeta^{\alpha_j}\e^{\xi_j})}=\prod_{j\in J} 
  \sum_{n_j=0}^{\infty} \zeta^{n_j\alpha_j} \e^{n_j\xi_j}.$$
  So $$Q(\a,f,T)(\vexi)=
  \sum_{\ve n \in \Z_{\geq0}^J} \Bigl(\sum_{\zeta \colon \zeta^{f}=1} \zeta^{\sum_j
    n_j\alpha_j-T}\Bigr)\e^{\sum_{j\in J} n_j\xi_j}.$$

We note that $\sum_{\zeta: \ \zeta^{f}=1}\zeta^m$ is zero except if $m\in \Z f$, when
this sum is equal to~$f$.  Then we obtain that
$Q({\a},f,T)$ is the sum over $\ve n\in \Z_{\geq0}^J$ such that     $\sum_j n_j \alpha_j-T\in \Z f$.
The equality $1=\sum_{j\in J} s_j \alpha_j+s_0 f$ implies that
 $T\equiv \sum_{j} t s_j\alpha_j $ modulo $f$, and  the condition
 $\sum_j n_j \alpha_j-T\in \Z f$ is equivalent to the condition
$\sum_{j}(n_j-Ts_j)\alpha_j \in \Z f$.

We see that the point  $\ve n - T\ve s$ is in the 
lattice $\Lambda(\a,f)$ as well as in the cone $-T\ve s+\R^J_{\geq0}$ (as $n_j\geq
0$). Thus the claim.
\end{proof}

By definition of the meromorphic functions $ S\bigl(-T \ve
  s+\R^J_{\geq0},\Lambda(\a,f)\bigr)(\vexi) $    and 
  $M\bigl(-T\ve s, \R_{\geq 0}^J,\Lambda(\a,f)\bigr)(\vexi), $
we obtain the following equality.

\begin{corollary}
$$Q({\a},f,T)(\vexi)= f\ M\bigl(-T\ve s, \R_{\geq 0}^J,\Lambda(\a,f)\bigr)(\vexi).$$
\end{corollary}

Using Remark~\ref{rem:restrict-to-SafT} we thus obtain by restriction to
$\vexi=\a_J x$ the following equality.

\begin{corollary} \label{add}
  $$\CF(\a,f,T)(x) =f\,  M\bigl(-T\ve s,\R_{\geq0}^J,\Lambda(\a,f)\bigr)(\a_J x)\prod_{j\colon f \mid \alpha_j}\frac{1}{1-\e^{\alpha_jx}}.$$
\end{corollary}

\subsection{Unimodular decomposition in the dual space}

The cone $\R_{\geq0}^J$ is in general not unimodular with respect to the
lattice $\Lambda(\a,f)$.  By decomposing $\R_{\geq0}^J$ in cones
$\coneU$ that are unimodular with respect to~$\Lambda(\a,f)$, modulo cones
containing lines, we can write
$$M\bigl(-T\ve s,\R_{\geq0}^J,\Lambda(\a, f)\bigr)=\sum_\coneU \epsilon_\coneU M(-T\ve
s,\coneU,\Lambda),$$
where $\epsilon_\coneU\in\{\pm1\}$. 
This decomposition can be computed using Barvinok's algorithm in polynomial
time for fixed~$k$ because the dimension $|J|$ is at most $k$.

\begin{remark}
For this particular cone and lattice, this decomposition modulo cones
containing lines is best done using the 
``dual'' variant of Barvinok's algorithm, as introduced
in \cite{BarviPom}.  This is in contrast to the
``primal'' variant described in \cite{Brion1997residue,koeppe-verdoolaege:parametric}; see also
\cite{so-called-paper-2} for an exposition of Brion--Vergne decomposition and
its relation to both decompositions. 
To explain this, let us determine the index of the cone~$\R_{\geq0}^J$ in the
lattice $\Lambda=\Lambda(\a,f)$;  the worst-case complexity of the signed cone
decomposition is bounded by a polynomial in the logarithm of this index. 

Let $B$ be a matrix whose columns form a basis of~$\Lambda$, so
$\Lambda=B\Z^J$.  Then $|\Z^J:\Lambda| = \mathopen|\det B\mathclose| = f$. By
Remark~\ref{rem:change-to-standard-lattice}, we find
$$
M\bigl(-T\ve s,\R_{\geq0}^J,\Lambda\bigr)(\vexi)
= M(-T B^{-1}\ve s,B^{-1} \R_{\geq0}^J, \Z^J)(B^\T\vexi).
$$
Let $\coneC$ denote the cone $B^{-1} \R_{\geq0}^J$, which is generated by the
columns of~$B^{-1}$.  Since $B^{-1}$ is not integer in general, we find
generators of~$\coneC$ that are primitive vectors of~$\Z^J$ by scaling each of
the columns by an integer.  Certainly $\mathopen|\det B\mathclose| B^{-1}$ is
an integer matrix, and thus we find that the index of the cone~$\coneC$ is
bounded above by $f^{r-1}$.  We can easily determine the exact index as
follows.  For each $i\in J$, the generator $\ve e_i$ of the original cone~$\R_{\geq0}^J$
needs to be scaled so as to lie in the lattice~$\Lambda$.  The smallest
multiplier $y_i\in\Z_{>0}$ such that $\langle \a_J, y_i \ve e_i\rangle \in \Z
f$ is $y_i = \lcm(\alpha_i, f) / \alpha_i$.  Thus the index of~$\R^J_{\geq0}$ in $\Z^J$ is the
product of the $y_i$, and finally the index of~$\R^J_{\geq0}$ in~$\Lambda$ is
$$\frac1{|\Z^r:\Lambda|}\prod_{i\in J} \frac{\lcm(\alpha_i, f)}{\alpha_i}
= \frac1f \prod_{i\in J} \frac{\lcm(\alpha_i, f)}{\alpha_i}.
$$

Instead we consider the dual cone, $\coneC^\circ = \{\, \veeta \in V^*_J : \ll
\veeta,\ve y\rr \geq 0\text{ for $\ve y\in \coneC$} \,\}$.  We have
$\coneC^\circ = B^{\T} \R^J_{\geq0}$.  Then the index of the dual cone
$\coneC^\circ$ equals~$\mathopen|\det B^\T\mathclose| = f$, which is much smaller than $f^{r-1}$.

Following \cite{DyerKannan97}, we now compute a decomposition of
$\coneC^\circ$ in cones~$\coneU^\circ$ that are unimodular with respect to~$\Z^J$,
modulo lower-dimensional cones,
\begin{alignat*}{2}
[\coneC^\circ] &\equiv \sum_\coneU \epsilon_\coneU [\coneU^\circ] &\quad& \text{(modulo
  lower-dimensional cones)}.
\intertext{Then the desired decomposition follows:}
[\coneC] &\equiv \sum_\coneU \epsilon_\coneU [\coneU] && \text{(modulo
  cones with lines)}.
\end{alignat*}
Because of the better bound on the index of the cone on the dual side, the
worst-case complexity of the signed decomposition algorithm is reduced.  This
is confirmed by computational experiments.

\end{remark}

\begin{remark}
  Although we know that the meromorphic function $M\bigl(-T\ve s,\R_{\geq0}^J,\Lambda(\a,f)\bigr)(\vexi)$
restricts via $\vexi=\a_J x$ to a meromorphic function of a single variable $x$,
it may happen that the individual functions 
$M\bigl(-T\ve s,\coneU,\Lambda(\a,f)\bigr)(\vexi)$ do not restrict.  In other words,
the line $\a_J x$ may be entirely contained in the set of
poles. If this is  the case, we can  compute  (in polynomial time) a regular
vector $\vebeta\in\Q^J$ so that, for $\epsilon\neq 0, $ the deformed vector $(\a_J+\epsilon \vebeta)x$ is not a pole of any of the  functions $M\bigl(-T\ve
s, \coneU,\Lambda(\a,f)\bigr)(\vexi)$ occurring.  
We then consider the meromorphic functions $\epsilon \mapsto M\bigl(-T\ve
s, \coneU,\Lambda(\a,f)\bigr)((\a_J+\epsilon \vebeta)x)$ and their Laurent
expansions at $\epsilon=0$ in the variable~$\epsilon$. We then add the constant terms of these expansions (multiplied by $\epsilon_\coneU$). This is the value of
$M\bigl(-T\ve s, \R_{\geq0}^J,\Lambda(\a,f)\bigr)(\vexi)$ 
at the point $\vexi=\a_J x$.
\end{remark}

\subsection{The periodic dependence in \boldmath$T$}

Now let us analyze the dependence in $T$ of 
the functions $M(-T\ve s,\coneU,\Lambda(\a,f))$, where $\coneU$ is a unimodular cone.  
Let the generators be $\ve g_i^\coneU$, so the elements $\ve g_i^\coneU$ form a basis
of the lattice $\Lambda(\a,f)$. Recall that the lattice $f \Z^r$ is contained
in $\Lambda(\a,f)$. Thus as $\ve s\in \Z^r$, we have 
$\ve s=\sum_i s_i \ve g_i^\coneU$ with $fs_i\in \Z$
and hence
$\smallstep{T\ve s}_\coneU=\sum_i \fractional{-Ts_i} \ve g_i^\coneU$
 with $\fractional{-Ts_i}$  a function of $T$ periodic modulo~$f$.

Thus the function $T\mapsto \smallstep{T\ve s}_\coneU$ is a  step linear function, modulo $f$, with value in $V$.
We then write
$$M(-T\ve s,\coneU,\Lambda(\a,f))(\vexi)= \e^{\langle\vexi,\fractional{T \ve
    s}_\coneU\rangle}\prod_{j=1}^r \frac{1}{1-\e^{\ll\vexi,\ve g_j\rr }}.$$
Recall that by Corollary \ref{add},
$$\CF(\a,f,T)(x) =f\,  M\bigl(-T\ve s,\R_{\geq0}^J,\Lambda(\a,f)\bigr)(\a_J x)\prod_{j\colon f \mid \alpha_j}\frac{1}{1-\e^{\alpha_jx}}.$$
Thus this is a meromorphic function of the variable $x$ of the form: $$\sum_\coneU \e^{l_\coneU(T)x} \frac{h(x)}{x^{N+1}},$$
where $h(x)$ is
holomorphic in $x$ and $l_\coneU(T)$ is a step linear function of $T$, modulo $f$.
Thus to compute
$$E_i(f)(T)=\res_{x=0}\frac{(- x)^i}{i!} \CF(\a,f,T)(x)$$
we only have to expand the function $x\mapsto \e^{l_\coneU(T)x}$  up to   the
power $x^{N-i}$.  This expansion can be done in polynomial time.
We thus see that, as stated in Theorem~\ref{E},
$E_i(f)(T)$ is a step polynomial of degree less than or equal to $N-i$,
which is periodic  of $T$ modulo~$f$.  This completes the proof of
Theorem~\ref{E} and thus the proof of Theorem~\ref{theo:complexity}.


\section{Periodicity of coefficients}
Now that we have the main algorithmic result we can prove some consequences to the description of the periodicity of the coefficients.  In this section, we determine the largest $i$  with a non-constant coefficient $E_i(t)$ and we give a polynomial time algorithm for computing it. This will complete the proof of Theorem \ref{theo:firstperiodico}.

\begin{theorem}\label{thm:mobiusFan}
Given as input a list of integers $\a=[\alpha_1, \dots, \alpha_{N+1}]$ with
their prime factorization $\alpha_i = p_1^{a_{i1}}p_2^{a_{i2}} \cdots
p_n^{a_{in}}$, there is a polynomial time algorithm to find all of the
largest sublists where the greatest common divisor is not
one. Moreover, if $\ell$ denotes the size of the largest sublists with greatest common divisor different from one, 
then (1) there are polynomially many such sublists, (2) the poset $\tilde{P}_{>\ell-1}$  is a fan (a poset with a maximal element and adjacent atoms), and
 (3) the M\"obius function for $P_{>\ell-1}$ is $\mu(f) =1$ if $G(f) \neq G(1)$ and $\mu(1) = 1 - (|\CG_{>\ell-1}(\a)|-1)$.
 \end{theorem}

\begin{proof} Consider the matrix $A=[a_{ij}]$. Let $c_{i_1},
\dots, c_{i_k}$ be column indices of $A$ that denote the columns that contain the largest number of non-zero elements among the columns. Let $\a^{(c_{i_j})}$  be the sublist of $\a$ that corresponds to the rows of $A$ where column $c_{i_j}$ has a non-zero entry. Each $\a^{(c_{i_j})}$ has greatest common divisor different from one. If $\ell$ is the size of the largest sublist of $\a$  with greatest common divisor different from one, then there are $\ell$ many $\alpha_i$'s that share a common prime. Hence each column $c_{i_1}$ of $A$ has $\ell$ many non-zero elements. Then each $\a^{(c_{i_j})}$ is a largest sublist where the greatest common divisor is not one. Note that more than one column index $c_i$ might produce the same sublist $\a^{(c_{i_j})}$. The construction of $A$, counting the non-zero elements of each column, and forming the sublist indexed by each $c_{i_j}$ can be done in polynomial time in the input size. 

To show the poset $\tilde{P}_{>\ell-1}$  is a fan, let $\CG =\{1, f_1, \dots, f_m \}$ be the set of greatest common divisors of
sublists  of size $>\ell -1$. Each $f_i$ corresponds to a greatest common divisor of a sublist $\a^{(i)}$ of $\a$ with size $\ell$. We cannot have $f_i \mid f_j$ for $i \neq j$ because if $f_i \mid f_j$, then  $f_i$ is also the greatest common divisor of $\a^{(i)} \cup \a^{(j)}$, a contradiction to the maximality of $\ell$. Then the M\"obius function is $\mu(f_i) = 1$, and $\mu(1) = 1-m.$ 

As an aside, $\gcd(f_i, f_j) = 1$ for all $f_i \neq f_j$ as if $\gcd(f_i, f_j) \neq 1$, then we can take the union of the sublist that produced $f_i$ and $f_j$ thereby giving a larger sublist with greatest common divisor not equal to one, a contradiction. 
\end{proof}

\begin{example} $[2^2 7^4 41^1, 2^1 7^2 11^1, 11^4, 17^3]$ gives the matrix
\[
\begin{pmatrix}
2 & 4 & 0 & 0 & 1 \\
1 & 2 & 1 & 0 & 0\\
0 & 0 & 4 & 0 & 0\\
0 & 0 & 0 & 3 & 0
\end{pmatrix}
\]
where the columns are the powers of the primes indexed by $(2, 7, 11, 17, 41)$. We see the largest sublists that have gcd not equal to one are $[2^2 7^4 41^1, 2^1 7^2 11^1]$ and $[2^1 7^2 11^1, 11^4]$. Then $\CG =\{1, 2^17^2, 11\}$. The poset $P_{>1}$ is
\begin{center}
\begin{tikzpicture}[scale=0.60,transform shape,->,>=stealth',shorten >=1pt,auto,node distance=3cm,
	  thick,main node/.style={circle,draw,font=\sffamily\Large\bfseries}]
	  \node[main node] (1) {$G(1)$};
	  \node[main node] (2) [below left of=1] {$G(2^17^2)$};
	  \node[main node] (3) [below  right of=1] {$G(11)$};
	  \node[main node] (4) [below left of = 3]{$\hat{0}$};
	
	  \path[every node/.style={font=\sffamily\small}]
		 (1) edge node  {} (2)
			edge node {} (3)
		(2) edge node  {} (4)
       (3) edge node  {} (4);
\end{tikzpicture}
\end{center}
and $\mu(1) = -1$, $\mu(11) = \mu(2^17^2) = 1$.
\end{example}

\begin{proof}[Proof of Theorem~\ref{theo:firstperiodico}]
 Let $\ell$ be the greatest integer for which there exists a sublist $\a_J$ with $|J|=\ell$, such that its gcd $f$ is not $1$. Then for $m \geq \ell$ the coefficient of  degree $m$, $E_m(T)$, is constant because in Equation \eqref{eq:Ehrhartcoeff},  $\CG_{>m}(\a) = \{1\}$. Hence $E_m(T)$ does not depend on $T$. We now focus on $E_{\ell -1}(T)$. To simplify Equation \eqref{eq:Ehrhartcoeff}, we first compute the $\mu(f)$ values.
 
  \begin{lemma}
For $\ell$ as in Theorem \ref{theo:firstperiodico}, the poset  $ \CG_{>\ell-1}(\a)$ is a fan, with one maximal element $1$ and adjacent elements $f$ which are pairwise coprime. In particular, $\mu(f)=1$ for $f\neq 1$.
    \end{lemma}
    \begin{proof}
        Let $\a_{J_1}$, $\a_{J_2}$ be two  sublists of length $\ell$ with gcd's $f_1\neq f_2$ both not equal to $1$. If $f_1$ and $f_2$ had a nontrivial common divisor $d$, then the list $\a_{J_1\cup J_2}$ would have a gcd not equal to $1$, in contradiction with its length being strictly greater than $\ell$.
    \end{proof}
 
Next we recall a fact about Fourier series and use it to show that each
term in the summation over $f\in \CG_{>\ell-1}(\a)$ in Equation
\eqref{eq:Ehrhartcoeff} has smallest period equal to~$f$.

\begin{lemma}\label{lemma:FourierPeriod}
        Let $f$ be a positive integer and let $\phi(t)$ be a periodic function on $\Z/f\Z$  with Fourier expansion
        \[
        \phi(t)=\sum_{n=0}^{f-1} c_n \e^{2i\pi \inlinefrac{nt}{f}}.
        \]
     If $c_n\neq 0$ for some $n$ which is coprime to $f$ then $\phi(t)$ has smallest period equal to $f$.
    \end{lemma}
    \begin{proof}
       Assume  $\phi(t)$ has period $m$ with $f=qm$ and $q>1$. We write its Fourier series as a function of period $m$.
      $$
        \phi(t)=\sum_{j=0}^{m-1} c'_j \e^{2i\pi \inlinefrac{jt}{m}}= \sum_{j= 0}^{m-1}c'_j \e^{2i\pi \inlinefrac{(jq)t}{f}}.
        $$
     By uniqueness of the Fourier coefficients, we have $c_n=0$  if $n$ is not a multiple of $q$ (and $c_{qj}=c'_j$). In particular, $c_n = 0$ if $n$ is coprime to $f$, a contradiction.
    \end{proof} 
Theorem~\ref{theo:firstperiodico} is thus the consequence of the following lemma.

   \begin{lemma}\label{lemma:fterm}
     Let $f\in \CG_{>\ell-1}(\a)$.
     The term in the summation over $f$ in \eqref{eq:Ehrhartcoeff} 
     has smallest period $f$ as a function of $T$.
   \end{lemma}
 \begin{proof}
    For $f=1$, the statement is clear. Assume $f \neq 1$. We observe that the $f$-term in \eqref{eq:Ehrhartcoeff} is a periodic function (of period $f$) which is \emph{given as the sum of its Fourier expansion} and is written as $\sum_{n=0}^{f-1}c_n \e^{-2i\pi \inlinefrac{nT}{f}}$  where
    $$
    c_n=- \res_{x=0}\frac{(-x)^{\ell -1}}{(\ell-1)!\,\prod_j\bigl(1-\e^{-2i\pi \inlinefrac{n \alpha_j}{f}}\e^{\alpha_j x}\bigr)}.
    $$
    Consider a coefficient for which  $n$  is coprime to $f$. We decompose the product according to whether $f$ divides $\alpha_j$ or not. The crucial observation is that there are exactly $\ell$ indices $j$ such that $f$ divides $\alpha_j$, because of the maximality assumption on $\ell$.   Therefore $x=0$ is a simple pole and the residue is readily computed. We obtain
    $$
   c_n= \frac{(-1)^{\ell-1}}{(\ell-1)!}
   \cdot\frac{1}{\prod_{j: f \nmid \alpha_j}\bigl(1-\e^{2i\pi \inlinefrac{n \alpha_j}{f} }\bigr)}
   \cdot\frac{1}{\prod_{j: f\mid\alpha_j}\alpha_j }.
    $$
    Thus $c_n\neq 0$ for an $n$ coprime with $f$. By Lemma \ref{lemma:FourierPeriod}, each $f$-term has minimal period $f$. 
 \end{proof} 
  As the various numbers~$f$ in  $\CG_{>\ell-1}(\a)$ different from $1$ are
  pairwise coprime and the corresponding terms have minimal period~$f$, $E_{\ell-1}(T)$ has minimal
  period $\prod\limits_{f \in \CG_{>\ell-1}(\a)}f > 1$.
  This completes the proof of Theorem~\ref{theo:firstperiodico}.
\end{proof}


\section{Summary of the algorithm 
  and computational experiments} \label{experiments}

In this last section we report on experiments using our algorithm. But first, let us review the key steps of the 
algorithm:

Given a sequence of integers $\a$ of length ${N+1}$, we wish to compute the
top $k+1$ coefficients of the quasi-polynomial $E(\a)(t)$ of degree $N$.  Recall that
$$E(\a)(t)=\sum_{i=0}^{N} E_{i}(t) t^{i}$$
where $E_{i}(t)$ is a periodic function of $t$ modulo some period $q_i$.  We
assume that greatest common divisor of the list $\a$ is~$1$.

\begin{enumerate}[\bf 1.]
\item 
  We have
$$E(\a)(t)= E_{\mathcal P_{> N-k}}(t)+  E_{\mathcal P_{\leq N-k}}(t)$$
with  $E_{\mathcal P_{\leq N-k}}(t)$  a periodic polynomial of degree strictly less
than $N-k$. Computing the first $k+1$ coefficients means to compute $
E_{\mathcal P_{> N-k}}(t).$ 

\item  By writing $[\mathcal P_{> N-k}]=\sum_{f\in \mathcal F_{>N-k}(\a)}\mu(f) [G(f)]$,
  we have
$$ E_{\mathcal P_{> N-k}}(t)=\sum_{f\in \mathcal F_{>N-k}(\a)} \mu(f)   E(f,\a)(t).$$

\item Fix $f$ an integer.
Write $$\CF(\a,f,T)(x)=\sum_{\zeta: \ \zeta^f=1}\frac{\zeta^{-T}}{\prod_{i=1}^{N+1}(1-\zeta^{\alpha_i}\e^{\alpha_ix})};$$
 $$E(f,\a)(t)=\sum_i t^i E_i(f)(t) \ \text{with}$$ $$E_i(f)(T)= \res_{x=0}
 (-x)^i/i! \cdot \CF(\a,f,T)(x).$$

\item  We fix $f$ and  let $r$ be the number of elements $\alpha_i$ such that $\alpha_i$ is
 not a multiple of $f$.

 We then list such $\alpha_i$ in the list $\a=[\alpha_1,\alpha_2,\ldots,\alpha_r]$.

 We  introduce a lattice $\Lambda:=\Lambda(\a,f)\subset \Z^r$ and an element
 $\ve s\in \Z^r$ so that $f\ve s\in \Lambda.$

 We decompose  the standard  cone $\R_{\geq0}^r$ as a signed decomposition, modulo cones containing lines, in
 unimodular cones $\coneU$ for the lattice $\Lambda$, obtaining
$$\CF(\a,f,T)(x)=\sum_{\coneU}\epsilon_\coneU 
M(T\ve s,\coneU,\Lambda)(\a_I x) \frac{1}{\prod_{i\colon f\mid
    \alpha_i}(1-\e^{\alpha_ix})}.$$

\item 
To compute $E_{N-i}(f)(T)$, we compute the Laurent series of $\CF(\a,f,T)(x)$ at
$x=0$ and take the coefficient in $x^{-N-1+i}$ of this Laurent series. As the
Laurent series of $\CF(\a,f,T)(x)$ starts by $x^{-N-1}$, if $i$ is less than $k$,
we just have to compute at most $k$ terms of this Laurent series.

\end{enumerate}


\subsection{Experiments} 

We first wrote a preliminary implementation of  our algorithm in \maple, which we call \mapleKnapsack in the following.
Later we developed a faster implementation in C++, which is referred to as \latteKnapsack in the following (we use the term 
knapsack to refer to the Diophantine problem  $\alpha_1x_1+\alpha_2 x_2+\cdots+\alpha_{N} x_{N}+\alpha_{N+1}x_{N+1}=t$).
Both implementations are released as part of the software package 
\latteintegrale \cite{latteintegrale}, version 1.7.2.\footnote{Available under the GNU General Public
  License at \url{https://www.math.ucdavis.edu/~latte/}.
  The \maple code \mapleKnapsack is also available separately at \url{https://www.math.ucdavis.edu/~latte/software/packages/maple/}.
} 

We report on two different benchmarks tests: 
\begin{enumerate}
\item We test the performance of the implementations 
\mapleKnapsack%
  \footnote{Maple usage:
    \maplecode{coeff\_Nminusk\_knapsack($\langle\mathit{knapsack\
        list}\rangle$, t, $\langle\mathit{k\ value}\rangle$)}.} and 
  \latteKnapsack \footnote{Command line usage:
    \shellcode{dest/bin/top-ehrhart-knapsack -f $\langle\mathit{knapsack\
        file}\rangle$ -o $\langle\mathit{output\ file}\rangle$ -k
      $\langle\mathit{k\ value}\rangle$}.},
  and also the implementation of the algorithm from \cite{so-called-paper-1},
  which refer to as \coneApx \footnote{Command line usage: \shellcode{dest/bin/integrate
      --valuation=top-ehrhart --top-ehrhart-save=\allowbreak$\langle\mathit{output\
        file}\rangle$ --num-coefficients=$\langle\mathit{k\ value}\rangle$
      $\langle\mathit{LattE\ style\ knapsack\ file}\rangle$}.},
	  on a collection of over $750$ knapsacks. The latter algorithm can compute the weighted Ehrhart quasi-polynomials for simplicial polytopes, and hence it is more general than the algorithm we present in this paper, but this is the only other available algorithm for computing coefficients
 directly. Note that the implementations of the \mapleKnapsack algorithm and the main computational part of the \coneApx algorithm are in \maple, making comparisons between the two easier. 

 \item Next, we run our algorithms on a few knapsacks that have been studied in the literature. We chose these examples because some of these problems are considered difficult in the literature.  We also present a comparison with other available software that can also compute information of the denumerant $E_\a(t)$:
   the codes \emph{CTEuclid6} \cite{guocexin} and \emph{pSn}
 \cite{sillszeilberger}.\footnote{Both codes can be downloaded from the
   locations indicated in the respective papers.  Maple scripts that
   correspond to our tests of these codes are 
   available at
   \url{https://www.math.ucdavis.edu/~latte/software/denumerantSupplemental/}.} 
 These codes use mathematical ideas that are different from those used in this paper. 
\end{enumerate}

All computations were performed on a 64-bit Ubuntu machine with 64 GB of RAM and eight Dual Core AMD Opteron 880 processors.

\subsection{\mapleKnapsack vs.\@ \latteKnapsack vs.\@ \coneApx}

Here we compare our two implementations with the \coneApx algorithm from \cite{so-called-paper-1}. We constructed a test set of  768 knapsacks. For each $3 \leq d \leq 50$, we constructed four families of knapsacks:
\begin{description}
	\item[random-3] Five random knapsacks in dimension $d-1$ where $a_1 = 1$ and the other coefficients, $a_2, \dots, a_{d}$, are $3$-digit random numbers picked uniformly
	\item[random-15] Similar to the previous case, but with a $15$-digit random number
	\item[repeat] Five knapsacks in dimension $d-1$ where $\alpha_1=1$ and all the other $\alpha_i$'s are the same $3$-digit random number. These produce few poles and have a simple poset structure. These are among the simplest knapsacks that produce periodic coefficients.
	\item[partition] One knapsack in the form $\alpha_i = i$ for $ 1 \leq i \leq d$.
\end{description}

For each knapsack, we successively compute the highest degree terms of the
quasi-polyno\-mial, with a time limit of $200$ CPU~seconds for each coefficient.  
Once a term takes longer than $200$ seconds to compute, we skip the remaining
terms, as they are harder to compute than the previous ones.  
We then count the maximum number of terms of the quasi-polynomial, starting from the
highest degree term (which would, of course, be trivial to compute), that can be
computed subject to these time limits.  
Figures \ref{fig:random3}, \ref{fig:random15}, \ref{fig:repeat},
\ref{fig:partition} show these maximum numbers of terms for the random-3,
random-15, repeat, and partition knapsacks, respectively. 
For example, in Figure \ref{fig:random3}, for each of the five random 3-digit
knapsacks in ambient dimension $50$, the \latteKnapsack method computed at
most $6$ terms of an Ehrhart polynomial, the \mapleKnapsack computed at most
four terms, and the \coneApx method computed at most the trivially computable highest
degree term. 

In each knapsack family, we see that each algorithm has a ``peak'' dimension
where after it, the number of terms that can be computed subject to the time limit quickly decreases; for the \latteKnapsack method, this is around dimension $25$ in each knapsack family. In each family, there is a clear order to which algorithm can compute the most: \latteKnapsack computes the most coefficients, while the \coneApx method computes the least number of terms.  In Figure \ref{fig:repeat}, the simple poset structure helps every method to compute more terms, but the two \maple scripts seem to benefit more than the \latteKnapsack method. 

Figure \ref{fig:partition} demonstrates the power of the LattE implementation. Note that a knapsack of this particular form in dimension $d$ does not start to have periodic terms until around $d/2$. Thus even though half of the coefficients are only constants we see that the \mapleKnapsack code cannot compute past a few periodic term in dimension $10$--$15$ while the \latteKnapsack method is able to compute the entire polynomial.

In  Figure \ref{fig:ratioRandom15} we plot the average speedup ratio between
the \mapleKnapsack and \coneApx implementations along with the maximum and
minimum speedup ratios (we wrote both algorithms  in \maple). The ratios are
given by the time it takes \coneApx to compute a term, divided by the time it
takes \mapleKnapsack to compute the same term, where both times are between
$0$ and $200$ seconds. For example, among all the terms computed in dimension
$15$ from random $15$-digit knapsacks, the average speedup between the two
methods was $8000$, the maximum ratio was $20000$, and the minimum ratio was
$200$. We see that in dimensions $3$--$10$, there are a few terms for which
the \coneApx method was faster than the \mapleKnapsack method, but this only
occurs for the highest degree terms. Also, after dimension $25$, there is
little variance in the ratios because the \coneApx method is only computing
the trivial highest term. Similar results hold for the other knapsack families, and so their plots are omitted. 

\begin{figure}
  \centering
    \includegraphics[width=0.95\textwidth]{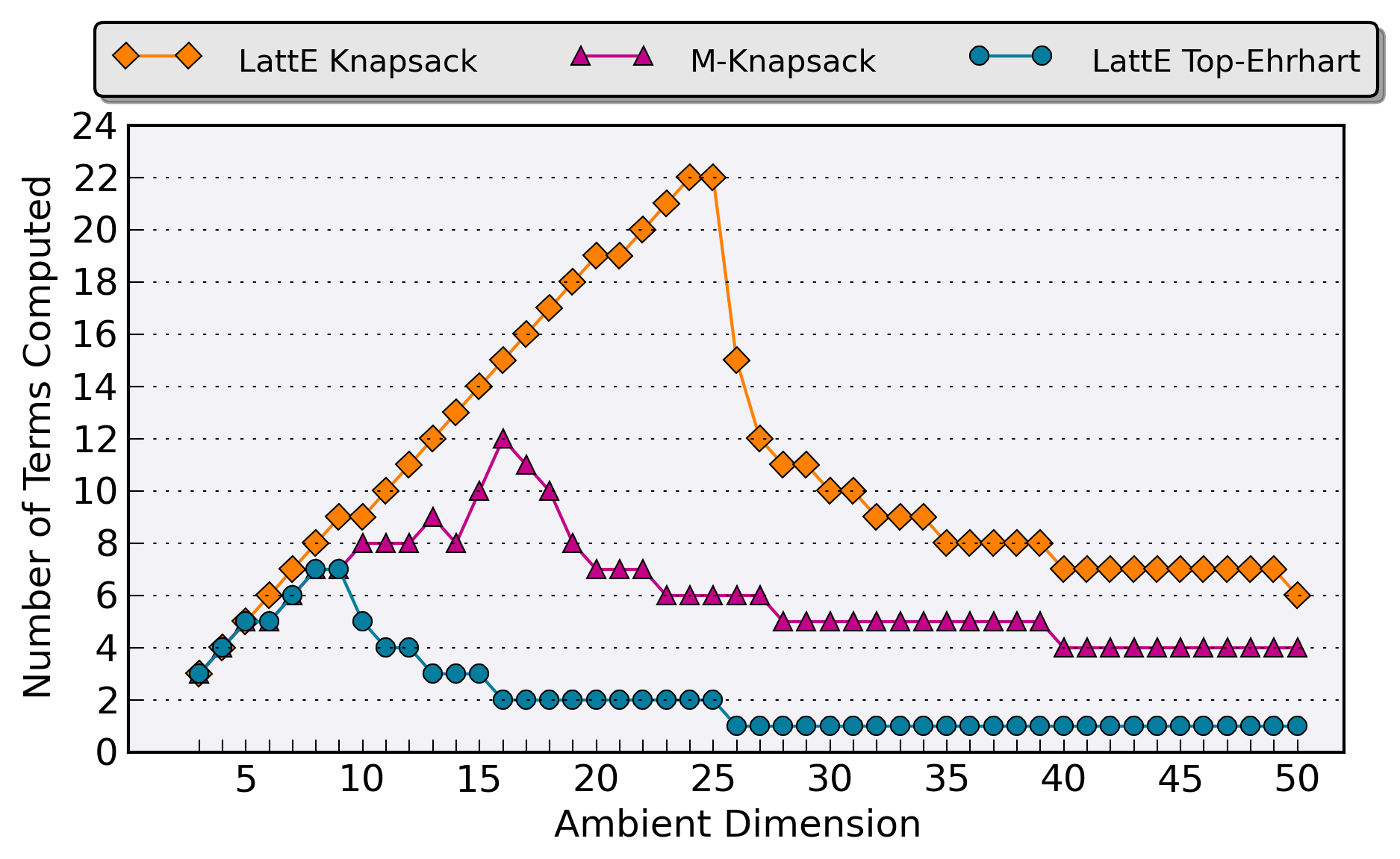}
 \caption{Random 3-digit knapsacks: Maximum number of coefficients each algorithm can compute where each coefficient takes less than 200 seconds.}
 \label{fig:random3}
\end{figure}

\begin{figure}
  \centering
    \includegraphics[width=0.95\textwidth]{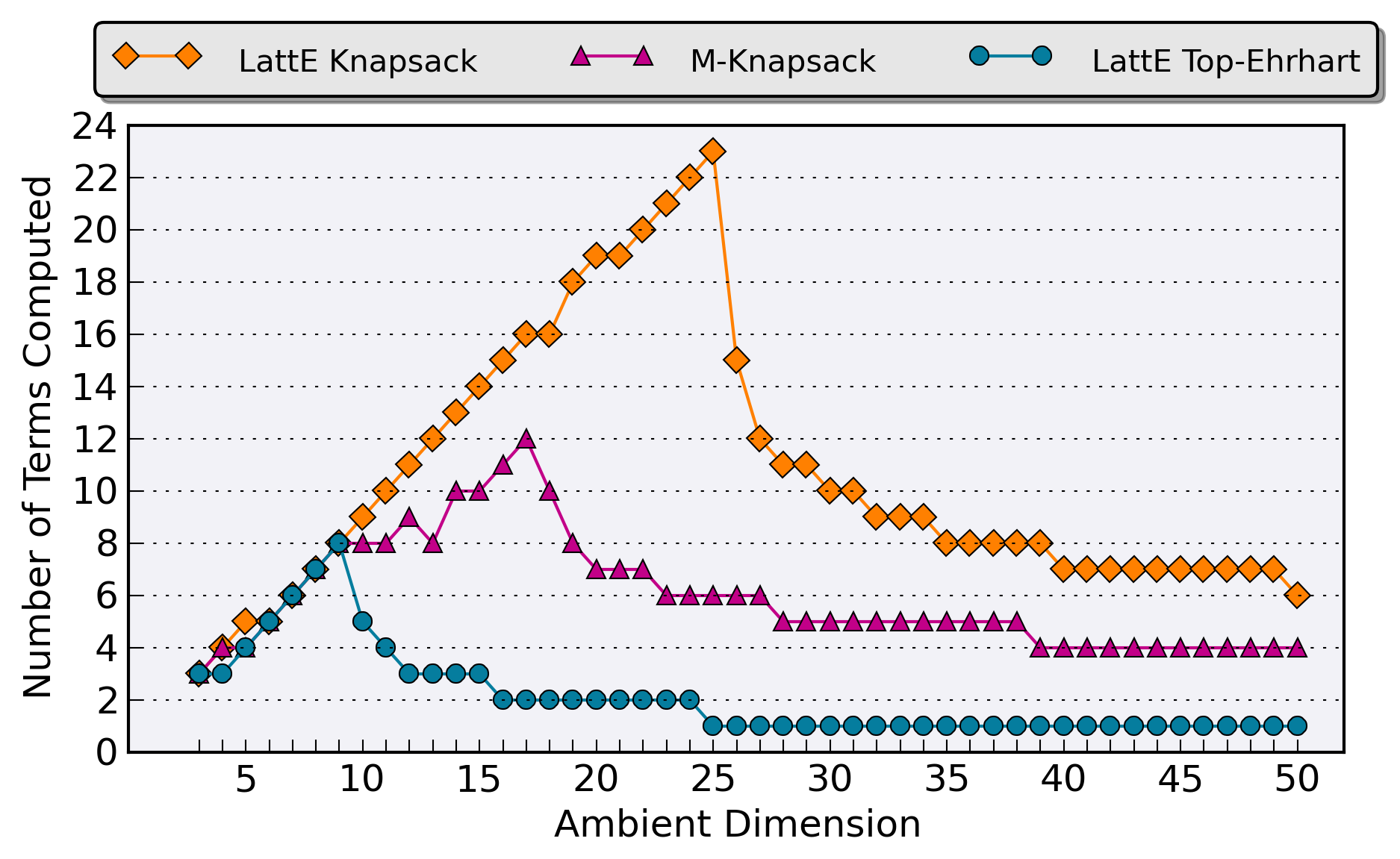}
 \caption{Random 15-digit knapsacks: Maximum number of coefficients each algorithm can compute where each coefficient takes less than 200 seconds.}
 \label{fig:random15}
\end{figure}

\begin{figure}
  \centering
    \includegraphics[width=0.95\textwidth]{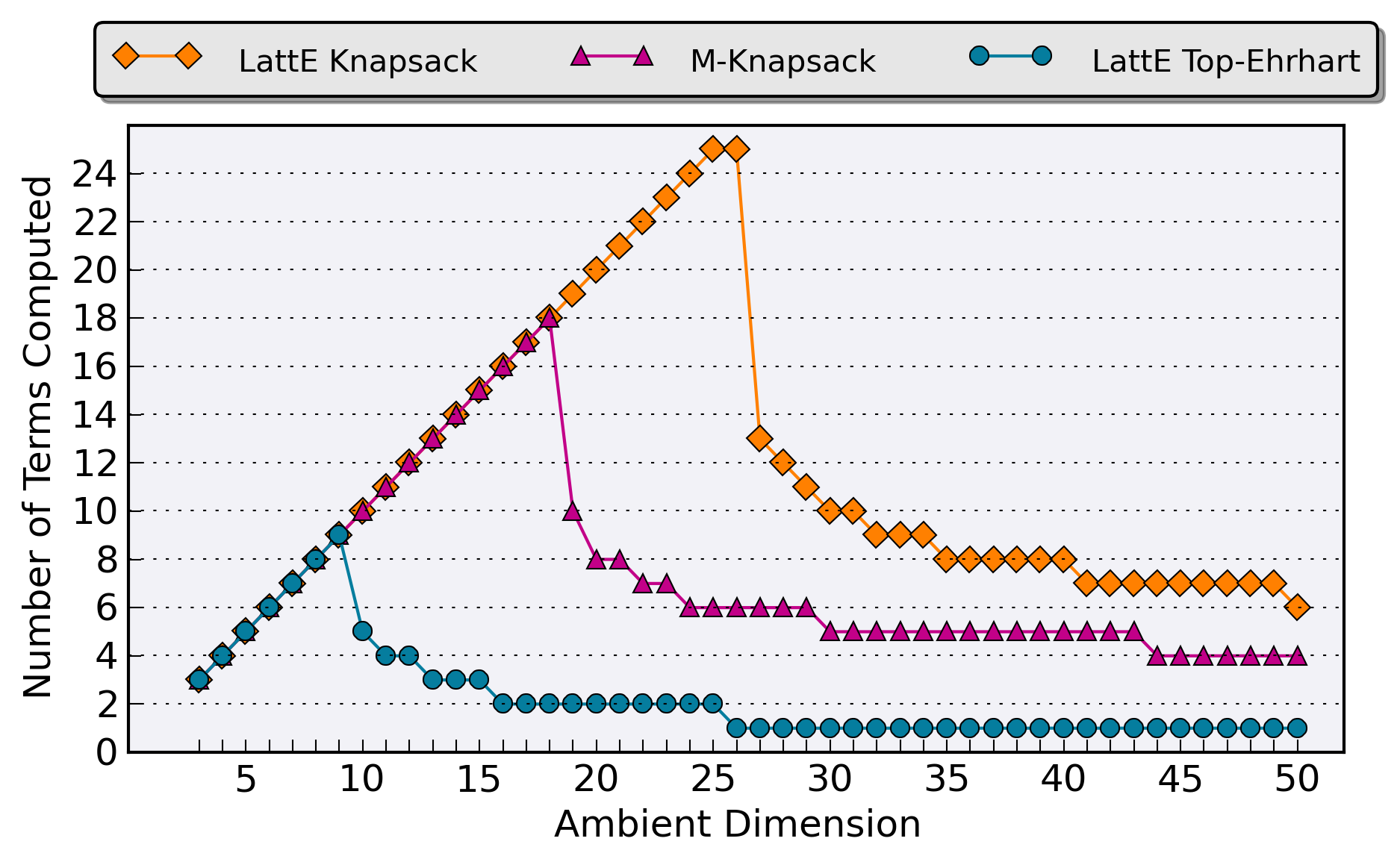}
 \caption{Repeat knapsacks: Maximum number of coefficients each algorithm can compute where each coefficient takes less than 200 seconds.}
    \label{fig:repeat}
\end{figure}

\begin{figure}
  \centering
    \includegraphics[width=0.95\textwidth]{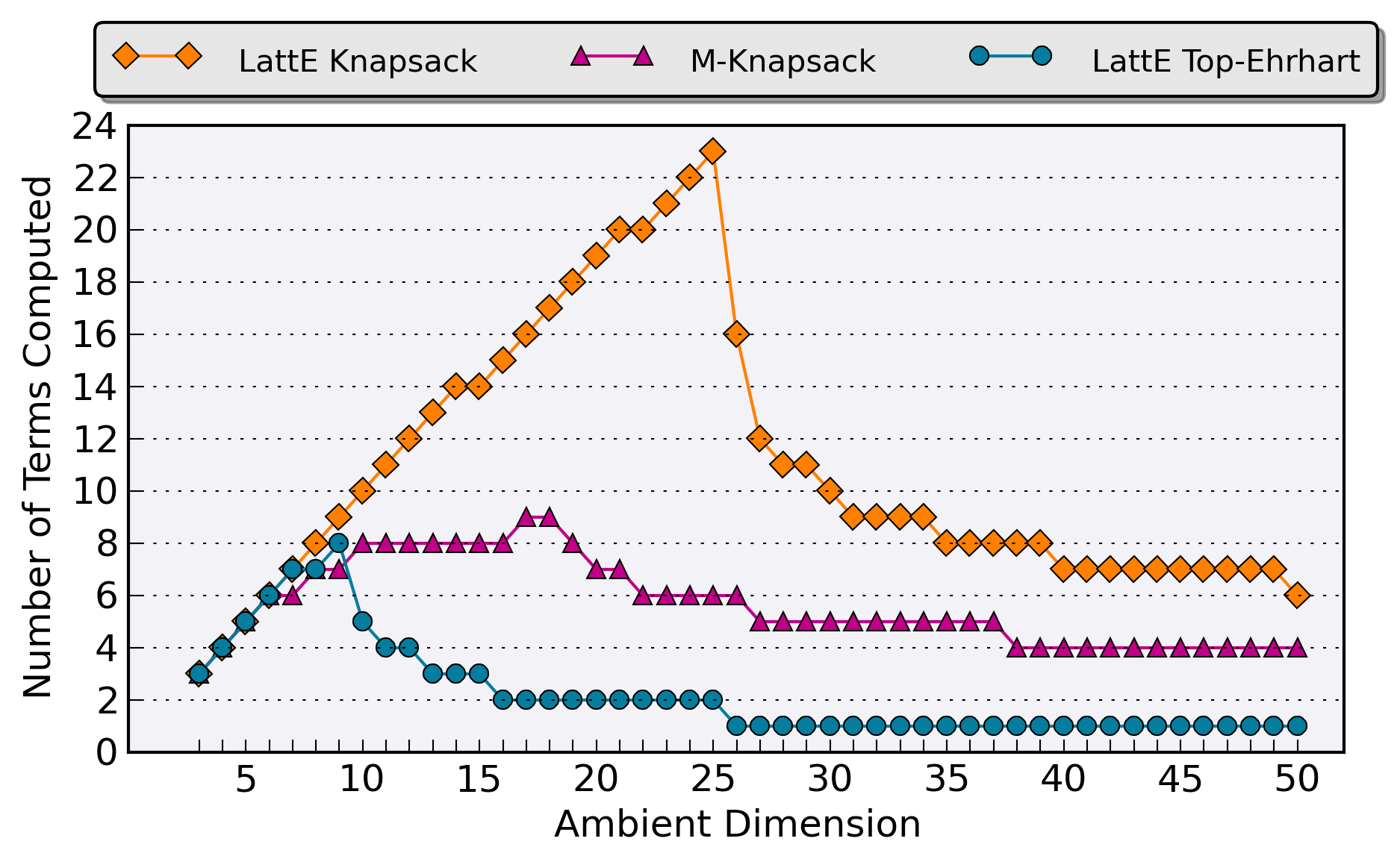}
 \caption{Partition knapsacks: Maximum number of coefficients each algorithm can compute where each coefficient takes less than 200 seconds.}
    \label{fig:partition}
\end{figure}

\begin{figure}
  \centering
    \includegraphics[width=0.95\textwidth]{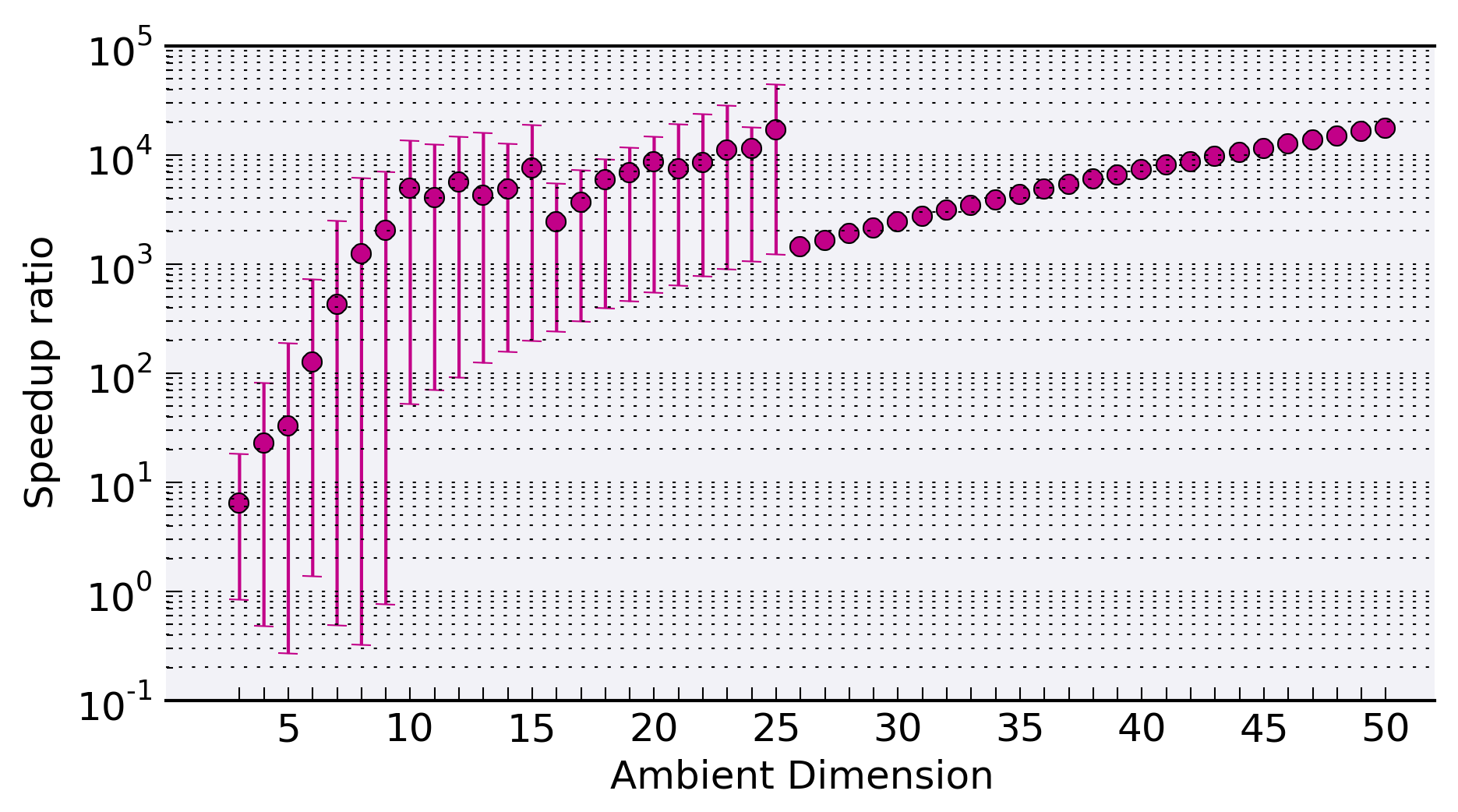}
 \caption{Average speedup ratio (dots) between the \mapleKnapsack and \coneApx codes along with maximum and minimum speedup ratio bounds (vertical lines) for the random 15-digit knapsacks.}
 \label{fig:ratioRandom15}
\end{figure}

\subsection{Other examples}

Next we focus on ten problems listed in Table \ref{tab:examples}. Some of these selected problems have been studied before in the literature \cite{aardallenstra,latte1,guocexin,GuoceXin2004}. Table \ref{tab:coneApxCTE} shows the time in seconds to compute the entire denumerant using the \mapleKnapsack, \latteKnapsack and \coneApx codes with two other algorithms: \emph{CTEuclid6} and \emph{pSn}.

 The \emph{CTEuclid6} algorithm \cite{guocexin}  computes the lattice point
 count of a polytope, and supersedes an earlier algorithm in
 \cite{GuoceXin2004}.\footnote{Maple usage:
   \maplecode{CTEuclid($F(\a)(x)/x^b$, t, [x])}; where $b = \alpha_1 + \cdots
   + \alpha_{N+1}$.}  
Instead of using Barvinok's algorithm to construct unimodular cones, the main idea used by the \emph{CTEuclid6} algorithm to find the constant term in the generating function $F(\a)(z)$ relies on recursively computing partial fraction decompositions to construct the series. Notice that the \emph{CTEuclid6} method only computes the number of integer points in one dilation of a polytope and not the full Ehrhart polynomial. We can estimate how long it would take to find the Ehrhart polynomial using an interpolation method by computing the time it takes to find one lattice point count times the periodicity of the polynomial and degree. Hence, in Table \ref{tab:coneApxCTE}, column ``one point'' refers to the running time of finding one lattice point count, while column ``estimate'' is an estimate for how long it would take to find the Ehrhart polynomial by interpolation. We see that the \emph{CTEuclid6} algorithm is fast for finding the number of integer points in a knapsack, but this would lead to a slow method for finding the Ehrhart polynomial. 
 
 The \emph{pSn} algorithm of \cite{sillszeilberger} computes the entire
 denumerant by using a partial fraction decomposition based
 method.\footnote{Maple usage:
   \maplecode{QPStoTrunc(pSn($\langle\mathit{knapsack\
       list}\rangle$,n,$j$),n)}; where $j$ is the smallest value in $\{100,
   200, \allowbreak 500, \allowbreak 1000, \allowbreak 2000, 3000\}$ that produces an answer.} More precisely the quasi-polynomials are represented as a function $f(t)$ given by $q$ polynomials $f^{[1]}(t), f^{[2]}(t), \dots,f^{[q]}(t)$ 
such that $f(t)=f^{[i]}(t)$ when $t \equiv i \pmod{q}$. To find the coefficients of the $f^{[i]}$ their method finds the  first few terms of the Maclaurin 
expansion of the partial fraction decomposition to find enough evaluations of those polynomials
and then recovers the coefficients of each the $f^{[i]}$ as a result of solving a linear system. This algorithm goes back to Cayley and it was  implemented in \maple. Looking at Table \ref{tab:coneApxCTE}, we see that the \emph{pSn} method is competitive with \latteKnapsack for knapsacks $1, 2, \dots, 6$, and beats \latteKnapsack in knapsack $10$. However, the \emph{pSn} method is highly sensitive to the number of digits in the knapsack coefficients, unlike our \mapleKnapsack and \latteKnapsack methods. For example,  the knapsacks $[1,2,4,6,8]$ takes 0.320 seconds to find the full Ehrhart polynomial, $[1,20,40, 60, 80]$ takes 5.520 seconds, and $[1, 200, 600, 900, 400]$ takes 247.939 seconds. Similar results hold for other three-digit knapsacks in dimension four. However, the partition knapsack $[1,2,3,\dots, 50]$ only takes 102.7 seconds. Finally, comparing the two \maple scripts, the \coneApx method outperforms the \mapleKnapsack method.

Table \ref{tab:coneApxCTE}  ignores one of the main features of our algorithm: that it can compute just the top $k$ terms of the Ehrhart polynomial. In Table \ref{tab:top3and4times}, we time the computation for finding the top three and four terms of the Ehrhart polynomial on the knapsacks in Table \ref{tab:examples}. We immediately see that our \latteKnapsack method takes less than one thousandth of a second in each example. Comparing the two \maple scripts, \mapleKnapsack greatly outperforms \coneApx. Hence, for a fixed $k$, the \latteKnapsack is the fastest method.

In summary, the \latteKnapsack is the fastest method for computing the top $k$ terms of the Ehrhart polynomial. The \latteKnapsack method can also compute the full Ehrhart polynomial in a reasonable amount of time up to  around dimension $25$, and the number of digits in each knapsack coefficient does not significantly alter performance. However, if the coefficients each have one or two digits, the \emph{pSn} method is faster, even in large dimensions. 

\begin{table}[t]
\centering
\caption{Ten selected instances}\label{tab:examples}
\begin{tabular}{ll}  
\toprule
Problem & Data                                      \\
\midrule
 \#1    & $[8,12,11]$                               \\
 \#2    & $[5,13,2,8,3]$                            \\
 \#3    & $[5,3,1,4,2]$                             \\
 \#4    & $[9,11,14,5,12]$                          \\
 \#5    & $[9,10,17,5,2]$                           \\
 \#6    & $[1,2,3,4,5,6]$                           \\
 \#7    & $[12223,12224,36674, 61119,85569]$        \\
 \#8    & $[12137, 24269,36405,36407,48545,60683]$  \\
 \#9    & $[20601,40429,40429,45415,53725,61919,64470,69340,78539,95043]$ \\
 \#10   & $[5, 10, 10, 2, 8, 20, 15, 2, 9, 9, 7, 4, 12, 13, 19]$ \\
\bottomrule
\end{tabular}
\end{table}

\begin{table}[t]
\centering\small
\caption{Computation times in seconds for finding the full Ehrhart polynomial
  using five different methods.}  \label{tab:coneApxCTE}
\begin{tabular}{
l
S[tabformat=1.3,tabnumalign=center]
S[tabformat=1.3,tabnumalign=center]
S[tabformat=1.3,tabnumalign=center]
S[tabformat=1.3,tabnumalign=center]
S[tabformat=1.3e+2,tabnumalign=center]
S[tabformat=1.3,tabnumalign=right]
}
\toprule
&  &  &  & \multicolumn{2}{c}{\emph{CTEuclid6}}  & \\
\cmidrule(l){5-6} 
& \latteKnapsack & \mapleKnapsack & \coneApx & {One point} &  {estimate} & \emph{pSn}\\
\midrule
\#1    & 0 & 0.316  & 0.160 & 0.004   & 3.168  &0.328\\
\#2    &0.03 & 5.984  & 2.208 & 0.048 & 347.4 & 0.292\\
\#3    &0.02 & 4.564  & 0.148 & 0.031 & 9.60   &0.212\\
\#4    &0.08 & 18.317 & 3.884 & 0.112 & 7761.6 &0.496\\
\#5    &0.06 & 15.200 & 3.588 & 0.096 & 734.4  &0.392\\
\#6    &0.11 & 37.974 & 8.068 & 0.088 & 31.68  &0.336\\
\#7    &0.19 & 43.006 & 8.424 & 0.436 & 9.466e+20 & {$>$30min}\\
\#8    &1.14 & 1110.857&184.663&2.120 & 8.530e+20 &{$>$30min}\\
\#9    &{$>$30min} & {$>$30min} & {$>$30min} & {$>$30min} & {$>$30min}  & {$>$30min} \\
\#10   &{$>$30min} & {$>$30min} & {$>$30min} & 142.792 & 1.333e+9 &2.336\\
\bottomrule
\end{tabular}
\end{table}


\begin{table}
\centering\small
\caption{Computation times in seconds for finding the top three and four  terms of the Ehrhart polynomial}  \label{tab:top3and4times}
\begin{tabular}
{lcc
S[tabformat=1.3,tabnumalign=center]
cc
S[obeyall,tabformat=3.3,tabtextalign=right,tabnumalign=right]
}  
\toprule

& \multicolumn{3}{c}{Top 3 coefficients}  & \multicolumn{3}{c}{Top 4 coefficients}   \\ 
\cmidrule(l){2-4}\cmidrule(l){5-7}
  & \emph{LattE} &  \mapleKnapsack & \emph{LattE} &  \emph{LattE} & \mapleKnapsack & \emph{LattE} \\
 & \emph{Knapsack} & & \emph{Top-Ehrhart} & \emph{Knapsack} & & \emph{Top-Ehrhart} \\
\midrule
\#1    & 0 & 0.305   & 0.128  & -- &  --   &  {--} \\
\#2    & 0 & 0.004   & 0.768  & 0  & 0.096 & 1.356\\
\#3    & 0 & 0.004   & 0.788  & 0  & 0.080 & 1.308\\
\#4    & 0 & 0.003   & 0.792  & 0  & 0.124 & 1.368\\
\#5    & 0 & 0.004   & 0.784  & 0  & 0.176 & 1.424\\
\#6    & 0 & 0.004   & 1.660  & 0  & 0.088 & 2.976\\
\#7    & 0 & 0.004   & 0.836  & 0  & 0.272 & 1.652\\
\#8    & 0 & 0.068   & 1.828  & 0  & 0.112 & 3.544\\
\#9    & 0 & 0.004   & 18.437 & 0  & 0.016 & 59.527\\      
\#10   & 0 & 0.012   & 142.104& 0  & 0.044 & 822.187\\
\bottomrule
\end{tabular}
\end{table}

\clearpage

\subsection*{Acknowledgments}
We are grateful to Doron Zeilberger and an anonymous referee for suggestions and comments.
The work for this article was done in large part during a SQuaRE program at the
American Institute of Mathematics, Palo Alto, in March 2012.  V.~Baldoni was partially supported by the 
Cofin 40\%, MIUR. De Loera was partially supported by NSF grant DMS-0914107, 
M.~K\"oppe was partially supported by NSF grant DMS-0914873.  B. Dutra was supported by the
NSF-VIGRE grant DMS-0636297. The support received is gratefully acknowledged.

\bibliographystyle{../amsabbrv}
\bibliography{../biblio}

\end{document}